\documentclass[reqno,12pt]{amsart} 
\usepackage{amsmath}
\usepackage{amssymb}
\usepackage{amsthm}
\usepackage{xcolor}
\usepackage{graphics}
\newcommand\NoBlackBoxes{\global\overfullrule0pt}
\NoBlackBoxes
\theoremstyle{plain} 
\newtheorem{theorem}{Theorem} 

\newtheorem{lemma}[theorem]{Lemma}
\newtheorem{proposition}[theorem]{Proposition}

\newtheorem{corollary}[theorem]{Corollary}
\def\4{\kern1pt}

\def\6{\vphantom0}

\def\8{\kern-10pt}
\def\7#1{_{(#1)}}

\theoremstyle{definition}
\newtheorem*{assumption (H1)}{Assumption (H1)}
\newtheorem*{assumption (H2)}{Assumption (H2)}

\theoremstyle{remark}

\numberwithin{equation}{section}
\numberwithin{theorem}{section}
\makeatletter
\let\serieslogo@\relax
\let\@setcopyright\relax

\def\speciallabelmark#1{\def\@currentlabel{#1}}
\makeatother

\newcommand{\Cal}{\mathcal}

\newcommand{\beqa}{\begin{eqnarray}}
\newcommand{\beqan}{\begin{eqnarray*}}
\newcommand{\eeqa}{\end{eqnarray}}
\newcommand{\eeqan}{\end{eqnarray*}}
\def\beq#1\eeq{\begin{equation}#1\end{equation}}

\begin{document}

\def\ffrac#1#2{\raise.5pt\hbox{\small$\4\displaystyle\frac{\,#1\,}{\,#2\,}\4$}}
\def\ovln#1{\,{\overline{\!#1}}}
\def\ve{\varepsilon}
\def\kar{\beta_r}

\title{Approximation of free convolutions by free infinitely divisible laws.}

\author{G. P. Chistyakov$^{1,2}$}
\address
{Gennadii Chistyakov \newline
Facult\"at f\"ur Mathematik \newline 
Universit\"at Bielefeld\newline 
Postfach 100131\newline
33501 Bielefeld\newline
Germany}
\email { chistyak@mathematik.uni-bielefeld.de} 

\author{F. G\"otze$^{1,2}$}
\thanks{1) Faculty of Mathematics , 
University of Bielefeld, Germany.}
\thanks{2) Research supported by CRC 1283.}
\address
{Friedrich G\"otze\newline
Fakult\"at f\"ur Mathematik\newline
Universit\"at Bielefeld\newline
Postfach 100131\newline
33501 Bielefeld \newline
Germany}
\email {goetze@mathematik.uni-bielefeld.de}

\date{October, 2020}


\maketitle
\markboth{ G. P. Chistyakov and F. G\"otze}{Approximations
in Free Probability Theory}

\begin{abstract}
Based on the~method of subordinating functions
we prove bounds for
the minimal error of approximations of $n$-fold convolutions of probability measures 
by free infinitely divisible probability measures.  
\end{abstract}

\section{Introduction}

In recent years a number of papers are investigating limit theorems
for the free convolution of probability measures 
introduced by D. Voiculescu.
The~key concept of this definition is the~notion of freeness,
which can be interpreted as a~kind of independence for
non commutative random variables. As in classical probability where
the~concept of independence gives rise to the~classical convolution, 
the~concept of freeness leads to a~binary operation on the~probability measures 
on the~real line, the~free convolution. Many classical results 
in the~theory of addition of independent random variables have their
counterpart in this new theory, such as the~law of large numbers,
the~central limit theorem, the~L\'evy-Khintchine formula and others.
We refer to Voiculescu, Dykema and Nica \cite{Vo:1992}, Hiai 
and Petz~\cite{HiPe:2000}, Nica and Speicher~\cite{NSp:2006} for a detailed introduction 
into these topics. 

In classical probability theory Doeblin~\cite{Do:1939} showed that it is possible to construct
independent identically distributed random variables $X_1,X_2,\dots$ such that the~distribution
of the~centered and normalized sum $b_{n_k}^{-1}(X_1+\dots+X_{n_k}-a_{n_k})$ does not converge to any
non degenerate distribution, for any choice of the~constants $a_n$ and $b_n$ and of sequences
$n_1<n_2<\dots$.
In view of these examples Kolmogorov~\cite{Ko:1953} 
suggested the class
of infinitely divisible distributions for approximating the sequence $\{\mu^{n*}\}_{n=1}^{\infty}$ of convolutions of some distribution $\mu$
in some metric as $n\to\infty$.
Starting with
Prokhorov~\cite{Pr:1955} and Kolmogorov~\cite{Ko:1956} this problem had been investigated intensively, culminating in the seminal results by  
Arak and Zaitsev~\cite{ArZa:1988}.

Due to the~Bercovici--Pata ~parallelism between limits laws for free and classical additive convolution~\cite{BeP:1999}
the~Doeblin results hold for free random variables and suggest that Kolmogorov approach is natural
in free Probability Theory as well.
Hence, we shall use this setup in free Probability Theory
and study in this paper approximations
of $n$-fold
additive free convolutions of probability measures by additive free infinitely divisible 
probability measures.

The~paper is organized as follows. In Section~2 we formulate and
discuss the main results of the~paper. In Section~3 we formulate
auxiliary results. In Section~4 we obtain lower bounds in this class of approximations.
We illustrate these
approximation results in Section 5 with a list of examples of probability measures which allow for exponential small approximation errors whereas in Section 6 we construct probability measures which only allow polynomial small approximation errors.

\section{Results}
Denote by $\mathcal M$ the~family of all Borel probability measures
defined on the~real line $\Bbb R$. On $\mathcal M$ define 
the~associative composition laws denoted $*$ and $\boxplus$ as follows.
For $\mu_1,\mu_2\in\mathcal M$ let a~probability measure $\mu_1*\mu_2$ denote
the~classical convolution of $\mu_1$ and $\mu_2$. 
In probabilistic terms, $\mu_1*\mu_2$
is the~probability distribution of $X+Y$, where $X$ and $Y$ are
(commuting) independent random variables with distributions $\mu_1$ 
and $\mu_2$, respectively. A~measure $\mu_1\boxplus\mu_2$ on the~other hand
denotes the~free (additive) convolution of $\mu_1$ and $\mu_2$ introduced by 
Voiculescu~\cite{Vo:1986} for compactly supported measures. 
Free convolution was extended by
Maassen~\cite{Ma:1992} to measures with finite variance and by Bercovici
and Voiculescu~\cite{BeVo:1993} to the~class $\mathcal M$.
Thus, $\mu_1\boxplus\mu_2$ is the~probability distribution of $X+Y$,
where $X$ and $Y$ are free random variables with distributions $\mu_1$ 
and $\mu_2$, respectively. 

Let $\Delta(\mu,\nu)$ be the~Kolmogorov distance between probability measures $\mu$ and $\nu$, i.e.,
$$
\Delta(\mu,\nu)=\sup_{x\in\mathbb R}|\mu((-\infty,x))-\nu((-\infty,x))|.
$$
In 1955 Prokhorov~\cite{Pr:1955} proved that for any $\mu\in\mathcal M$
\begin{equation}\label{2.1}
\Delta(\mu^{n*},\bold D^*):=\inf_{\nu\in\bold D^*}\Delta(\mu^{n*},\nu)\to 0,\quad n\to\infty,
\end{equation}
where $\mu^{n*}$ is the~$n$-fold convolution of the~probability measure $\mu$ and $\bold D^*$ 
is the~set of classical infinitely divisible probability measures. Kolmogorov~\cite{Ko:1956} noted that
convergence in (\ref{2.1}) to zero is uniform with respect to $\mu$ throughout the~class $\mathcal M$.
Work by a~number of researches (a~detailed history of the~problem may be found in \cite{ArZa:1988})
eventually proved upper and lower bounds for the~function $\psi(n):=\sup_{\mu\in\mathcal M}\Delta(\mu^{n*},\bold D^*)$.
A~final answer was given by Arak~\cite{Ar:1981}, \cite{Ar:1982}, who proved the~following bound:
$c_1n^{-2/3}\le\psi(n)\le c_2n^{-2/3}$, where $c_1$ and $c_2$ are absolute positive constants.
Extending Prohorov's results~\cite{Pr:1955} in \cite{Ch:1995} the~problem of determining the~possible
rate of decrease of $\Delta(\mu^{n*},\bold D^*)$ to zero as $n\to\infty$ has been studied for probability measures $\mu\not\in\bold D^*$.

Let $\mu\in\mathcal M$, denote $\mu^{n\boxplus}:=\mu\boxplus\dots\boxplus\mu$ ($n$ times). 
Recall that $\mu\in\mathcal M$ is $\boxplus$-infinitely divisible
if, for every $n\in\mathbb N$, there exists $\mu_n\in\mathcal M$ such that $\mu=\mu_n^{n\boxplus}$. In the sequel we will write
in this case that $\mu\in\bold D^{\boxplus}$. As in the classical case introduce the~quantity
$$
\Delta(\mu^{n\boxplus},\bold D^{\boxplus}):=\inf_{\nu\in\bold D^{\boxplus}}\Delta(\mu^{n\boxplus},\nu)
$$
and raise the~question of the~behaviour of this quantity when $n\to\infty$.

Let $\Bbb C^+\,(\Bbb C^-)$ denote the~open upper (lower) half of
the~complex plane. For $\mu\in\mathcal M$,
define its Cauchy transform by
\begin{equation}\label{2.2}
G_{\mu}(z)=\int\limits_{-\infty}^{\infty}\frac {\mu(dt)}{z-t},
\qquad z\in\Bbb C^+.
\end{equation}
It is clear that the~quantity $c_1(\mu):=\Im (1/G_{\mu}(i))-1$ is positive if $\mu\ne\delta_b$ with $b\in\mathbb R$ and
$\delta_b$ denoting a Dirac measure concentrated at the~point $b$. In the~sequel we denote by
$c(\mu),c_2(\mu),c_3(\mu),\dots$ positive constants depending on $\mu$ only. By~$c(\mu)$ we 
denote constants  in different (or even in the same) formulae. 

We proved in~\cite{ChG:2013}
the following result.

\begin{theorem}\label{th1}
Let $\mu\in\mathcal M$ and $c_1(\mu)>0$. Then
\begin{equation}\notag
\Delta(\mu^{n\boxplus},\bold D^{\boxplus})\le c(\mu)\Big(\frac 1{\sqrt n}\int_{[-N_n,N_n]}|u|\mu(du)+\mu(\{\mathbb R\setminus[-N_n/8,N_n/8]\}\Big),
\quad n\in\mathbb N, 
\end{equation}
where $N_n:=\sqrt{\sigma_1(\mu)(n-1)}$.
\end{theorem}

From this theorem the~free analogue of Prokhorov's result (\ref{2.1}) immediately follows.
\begin{corollary}\label{cor1}
$$
\Delta(\mu^{n\boxplus},\mathbf D^{\boxplus})\to 0,\quad \text{as}\quad n\to\infty.
$$ 
\end{corollary}
Denote by $\mathcal M_d,\,d \ge 0$, the~set of
probability measures such that
$\beta_{d}(\mu):=\int\limits_{\mathbb R}|x|^{d}\,\mu(dx)<\infty$.
We easily obtain from Theorem~\ref{th1} the~following upper bound. 
\begin{corollary}\label{cor2}
Let $\mu\in\mathcal M_{d}$ with some $d>0$. Then
$$
\Delta(\mu^{n\boxplus},\bold D^{\boxplus})\le c(\mu)n^{-\min\{d,1/2\}},\quad n\in\mathbb N.
$$ 
\end{corollary}

Now we establish a~possible rate of decrease to zero as $n\to\infty$ for $\mu\not\in \mathbf{D}^{\boxplus}$.
Denote $\mathbb N_1=\mathbb N\setminus\{1,2\}$.
\begin{theorem}\label{th2}
Let $\mu\in\mathcal M_{d},\,d\in\mathbb N_1$,
and $\mu\not\in \mathbf{D}^{\boxplus}$. Then 
$$
\liminf_{n\to\infty}\frac{\log\Delta(\mu^{n\boxplus},\mathbf D^{\boxplus})}{\sqrt n}>-\infty.
$$ 
\end{theorem}

Corollary~\ref{cor2} and Theorem~\ref{th2} provide a two-sided estimate of the~quantity
$\Delta(\mu^{n\boxplus},\mathbf D^{\boxplus})$ for the probability measures $\mu\in\mathcal M_{d},d\in\mathbb N_1$ we considered, 
and $\mu\not\in \mathbf{D}^{\boxplus}$
\begin{equation}\label{2.3}
e^{-c_2(\mu)\sqrt n}\le \Delta(\mu^{n\boxplus},\mathbf D^{\boxplus})\le \frac{c_3(\mu)}{\sqrt n},\quad n\in\mathbb N.
\end{equation}

Denote by $m_s(\mu):=\int_{\Bbb R}u^s\,\mu(du),\,s=0,1,\dots$, moments of the~probability measure $\mu$.
For $\mu\in\mathcal M_{d}$ with $d\ge 3$, we can improve the~upper bound in (\ref{2.3}). 
Let $\mu\in\mathcal M_{d}$ with $d\ge 3$ and $m_1(\mu)=0,m_2(\mu)=1$.
Denote $\mu_n((-\infty,x)):=\mu((-\infty,x\sqrt {n}))$, $x\in\Bbb R$.

Consider the~family of $\boxplus$-infinitely divisible probability-measures $\{w_a:a\in\mathbb R\}$
with the~Cauchy transform
\begin{equation}\label{2.3h}
G_{w_a}(z)=\Big(a+\frac 12\Big(z-a+\sqrt{(z-a)^2-4}\Big)\Big)^{-1},\quad z\in\mathbb C.
\end{equation}
These probability measures are a special case of the~free centered (i.e. with mean zero) Meixner measures. 
In this formula the~branch of the~analytic square root should be chosen according to the~condition  $\Im z>0\implies \Im (1/G_{w_a}(z))\ge 0$.

In~\cite{ChG:2013a} it was proved for $\mu\in\mathcal M_d$ with $d\ge 3$ that 
there exists a~positive absolute constant $c$ such that
\begin{equation}\label{2.4}
\Delta(\mu_n^{n\boxplus},w_{a_n})
\le c\frac{\beta_{d_1}(\mu)}{n^{(d_1-2)/2}},\quad n\in\mathbb N,
\end{equation} 
where $a_n:=m_3(\mu)/\sqrt n$ and $d_1:=\min\{d,4\}$.

The estimate (\ref{2.4}) shows that for $\mu\in\mathcal M_d$ with $d\ge 3$
the~upper bound in (\ref{2.3}) has the~form
\begin{equation}\notag
\Delta(\mu^{n\boxplus},\mathbf D^{\boxplus})\le c\frac{\beta_{d_1}(\mu)}{n^{(d_1-2)/2}},\quad n\in\mathbb N.
\end{equation}

In the~sequel we denote by $\varphi(t;\mu)$ the~characteristic function of the~probability measure $\mu$, i.e., $\varphi(t;\mu)
=\int\limits_{\mathbb R}e^{itx}\,\mu(dx),\,t\in\mathbb R$.
We obtain the~statement of Theorem~\ref{th2} as a~consequence of the~next result.
\begin{theorem}\label{th3}
Let $\mu\in\mathcal M_{d},\,d\in\mathbb N_1$, and the~following relation is sutisfied:
\begin{equation}\notag
\liminf_{n\to\infty}\frac{\log\Delta(\mu^{n\boxplus},\mathbf D^{\boxplus})}{\sqrt n}=-\beta<0.
\end{equation}
Then $\varphi(t;\mu)=\varphi(t;\rho),\,t\in[-\frac{\beta}{300\sqrt{m_2(\mu)-m_1^2(\mu)}},\frac{\beta}{300\sqrt{m_2(\mu)-m_1^2(\mu)}}]$, where $\rho$ is some $\boxplus$-infinitely divisible probability measure. 
\end{theorem}

In one's turn Theorem~\ref{th3} is a~simple consequence of the~following theorem.
\begin{theorem}\label{th4}
Let $\mu\in\mathcal M_{d},\,d\in\mathbb N_1$, and there exists a~sequence $\{\rho_n\}_{n=1}^{\infty}$ of probability measures such that
\begin{equation}\notag
\liminf_{n\to\infty}\frac{\log\Delta(\mu^{n\boxplus},\rho_n^{n\boxplus})}{\sqrt n}=-\beta<0.
\end{equation}
Then there exists a~subsequence of positive integers $\{n'\}$ such that $\rho_{n'}$ converges weakly to some probability measure $\rho$ and $\varphi(t;\mu)=\varphi(t;\rho),\,t\in[-\frac{\beta}{300\sqrt{m_2(\mu)-m_1^2(\mu)}},\frac{\beta}{300\sqrt{m_2(\mu)-m_1^2(\mu)}}]$.
\end{theorem}
\begin{corollary}\label{cor3}
Let $\mu$ and $\rho$ belong to $\mathcal M_{d},\,d\in\mathbb N_1$, and
\begin{equation}\notag
\liminf_{n\to\infty}\frac{\log\Delta(\mu^{n\boxplus},\rho^{n\boxplus})}{\sqrt n}=-\beta<0.
\end{equation}
Then $\varphi(t;\mu)=\varphi(t;\rho),\,t\in[-\frac{\beta}{300\sqrt{m_2(\mu)-m_1^2(\mu)}},\frac{\beta}{300\sqrt{m_2(\mu)-m_1^2(\mu)}}]$.
\end{corollary}

Consider the~class $\mathcal M_{det}$ of probability measures $\mu$ which are uniquely determined by their moment sequence $\{m_k(\mu)\}_{k=0}^{\infty}$, for details see the
introduction of\cite{Akh:1965}. 
Then the~following theorem holds.
\begin{theorem}\label{th5}
Let $\mu\in\mathcal M_{det}$ and $\mu\not\in \mathbf{D}^{\boxplus}$. 
Then 
\begin{equation}\label{2.5*}
\liminf_{n\to\infty}\frac{\log\Delta(\mu^{n\boxplus},\mathbf D^{\boxplus})}{\sqrt n}=0.
\end{equation}
\end{theorem}
Since the~class $\mathcal M_{C}$ of probability measures $\mu$ such that  Carleman's condition
\begin{equation}\notag
\sum_{k=1}^{\infty}\frac 1{\root {2k}\of{m_{2k}(\mu)}}=\infty, 
\end{equation}
holds is contained in $\mathcal M_{det}$ we note that relation(\ref{2.5*}) holds for $\mu\in\mathcal M_{C}$.

The next Theorem~2.9 shows that Theorem~2.4 gives a lower bound which is close to the optimal one. Before to formulate Theorem~2.9 we remark
that in the sequel we use slowly varying positive functions $L(x)\ge 1$ with continuous derivative, satisfying the following property: $L(x)\sim L(x/L(x))$ as $x\to\infty$. 
\begin{theorem}\label{th7}
There exists a~symmetric probability measure $\mu\not\in \mathbf{D}^{\boxplus}$ such that
\begin{equation}\label{2.6}
\Delta(\mu^{n\boxplus},\mathbf D^{\boxplus})\le e^{-c_4(\mu)\sqrt n/L(n)},\quad n\in\mathbb N, 
\end{equation}
where $L(n)\ge 1$ is a~slowly varying positive function such that
\begin{equation}
\sum_{n=1}^{\infty}\frac 1{nL(n)}<\infty.
\end{equation}
\end{theorem}

Let $\mu\in\mathcal M_{2k}$ with some positive integer $k\ge 1$. Denote 
by $\alpha_s(\mu),\,s=1,\dots,2k$, free cumulants of of the~probability measure $\mu\in\mathcal M_{2k}$ 
(see the~definition of free cumulants in Section~3).

\begin{theorem}\label{th6}
Let $\mu\in\mathcal M_{2k},\,k\in\mathbb N_1$, with $\alpha_{2k}(\mu)< 0$.
Assume that there exists a~$\boxplus$-infinitely divisible probability measure $\rho\in \mathcal M_{2k}$
such that $\alpha_s(\mu)=\alpha_s(\rho),\,s=1,\dots,2k-1$. Then there exist 
positive constants $c_5(\mu,\rho),c_6(\mu,\rho)$ depending on $\mu$ and $\rho$ only
such that
\begin{equation}\label{2.5}
\frac {c_5(\mu,\rho)}{n^{k+2+3/(k-1)}}\le
\Delta(\mu^{n\boxplus},
\rho^{n\boxplus})
\le \min\Big(1,\frac {c_6(\mu,\rho)}{n^{(k-4)/4}}\Big),\quad n\in\mathbb N. 
\end{equation}
\end{theorem}
 
\section{Auxiliary results}

We need results about some classes of analytic functions
(see {\cite{Akh:1965}, Section~3).

The~class $\mathcal N$ (Nevanlinna, R.) is the~class of analytic 
functions $f(z):\mathbb C^+\to\{z: \,\Im z\ge 0\}$.
For such functions there is the~integral representation,
for $z\in\mathbb C^+$,
\begin{equation}\label{3.1}
f(z)=a+bz+\int\limits_{\mathbb R}\frac{1+uz}{u-z}\,\tau(du)=
a+bz+\int\limits_{\mathbb R}\Big(\frac 1{u-z}-\frac u{1+u^2}\Big)(1+u^2)
\,\tau(du),
\end{equation}
where $b\ge 0$, $a\in\mathbb R$, and $\tau$ is a~non negative finite
measure. Moreover, $a=\Re f(i)$ and $\tau(\mathbb R)=\Im f(i)-b$.   
From this formula it follows that 
\begin{equation}\label{3.2}
f(z)=(b+o(1))z
\end{equation} 
for $z\in\mathbb C^+$
such that $|\Re z|/\Im z$ stays bounded as $|z|$ tends to infinity (in other words
$z\to\infty$ non-tangentially to $\mathbb R$).
Hence if $b\ne 0$, then $f$ has a~right inverse $f^{(-1)}$ defined
on the~region 
$$
\Gamma_{\alpha,\beta}:=\{z\in\mathbb C^+:|\Re z|<\alpha \Im z,\,\Im z>\beta\}
$$
for any $\alpha>0$ and some positive $\beta=\beta(f,\alpha)$.

A~function $f\in\mathcal N$ admits the~representation
\begin{equation}\label{3.3}
f(z)=\int\limits_{\mathbb R}\frac{\sigma(du)}{u-z},\quad z\in\mathbb C^+,
\end{equation}
where $\sigma$ is a~finite non negative measure, if and only if
$\sup_{y\ge 1}|yf(iy)|<\infty$.

The~Stieltjes-Perron inversion formula for the~functions $f$ of class 
$\mathcal N$ has the~following form.
Let $\psi(u):=\int_0^u(1+t^2)\,\tau(dt)$. Then
\begin{equation}\label{3.4}
\psi(u_2)-\psi(u_1)=\lim_{\eta\to 0}\frac 1{\pi}\int\limits_{u_1}
^{u_2}\Im f(\xi+i\eta)\,d\xi,
\end{equation}
where $u_1<u_2$ denote two continuity points of the~function $\psi(u)$.

For $\mu\in\mathcal M$, consider its Cauchy transform $G_{\mu}(z)$
(see (\ref{2.2})).
Following Maassen~\cite{Ma:1992} and Bercovici and 
Voiculescu~\cite{BeVo:1993}, we shall consider in the~following
the~ {\it reciprocal Cauchy transform}
\begin{equation}\label{3.5}
F_{\mu}(z)=\frac 1{G_{\mu}(z)}.
\end{equation}
The~corresponding class of reciprocal Cauchy
transforms of all $\mu\in\mathcal M$ will be denoted by $\mathcal F$.
This class coincides with the~subclass of Nevanlinna functions $f$
for which $f(z)/z\to 1$ as $z\to\infty$ non-tangentially to $\mathbb R$.
Indeed, reciprocal Cauchy transforms of probability
measures have obviously such 
property. Let $f\in\Cal N$ and $f(z)/z\to 1$ as $z\to\infty$ non-tangentially 
to $\mathbb R$. Then, by (\ref{3.2}), $f$ admits the~representation~(\ref{3.1}) with $b=1$.
By (\ref{3.2}) and (\ref{3.3}), $-1/f(z)$ admits 
the~representation~(\ref{3.3}) with $\sigma\in\Cal M$.

The~functions $f$ of the~class $\mathcal F$ satisfy the~inequality
\begin{equation}\label{3.5**}
\Im f(z)\ge \Im z,\qquad z\in\mathbb C^+.
\end{equation}

The~function $\phi_{\mu}(z)=F_{\mu}^{(-1)}(z)-z$ is called
the~Voiculescu transform of $\mu$ and
$\phi_{\mu}(z)$ is an~analytic function on $\Gamma_{\alpha,\beta}$ 
with the~property 
$\Im \phi_{\mu}(z)\le 0$ for $z\in \Gamma_{\alpha,\beta}$, where 
$\phi_{\mu}(z)$ is defined. 
On the~domain $\Gamma_{\alpha,\beta}$, where the~functions $\phi_{\mu_1}(z)$,
$\phi_{\mu_2}(z)$, and $\phi_{\mu_1\boxplus\mu_2}(z)$ are defined, we have
\begin{equation}\label{3.6}
\phi_{\mu_1\boxplus\mu_2}(z)=\phi_{\mu_1}(z)+\phi_{\mu_2}(z).
\end{equation} 
This relation  for the~distribution $\mu_1\boxplus\mu_2$ of $X+Y$,
where $X$ and $Y$ are free random variables, is due to
Voiculescu~\cite{Vo:1986}
for the case of  compactly supported measures.
The~result was extended by Maassen~\cite{Ma:1992} to measures 
with finite variance; the~general case was proved by 
Bercovici and Voiculescu~\cite{BeVo:1993}.

Assume that $\beta_k(\mu)<\infty$ for some $k\in\mathbb N$. Then
\begin{equation}\label{3.7a}
G_{\mu}(z)=\frac 1z+\frac{m_1(\mu)}{z^2}+\dots+\frac{m_k(\mu)}{z^{k+1}}
+o\Big(\frac 1{z^{k+1}}\Big),\quad z\to\infty,\,\,z\in\Gamma_{\alpha,1}.
\end{equation}
It follows from this relation (see for example \cite{Ka:2007a}) that
\begin{equation}\label{3.7}
\phi_{\mu}(z)=\alpha_1(\mu)+\frac{\alpha_2(\mu)}{z}+\dots
+\frac{\alpha_{k}(\mu)}{z^{k-1}}+o\Big(\frac 1{z^{k-1}}\Big),
\quad z\to\infty,\,\,z\in\Gamma_{\alpha,1}.
\end{equation}
We call the~coefficients $\alpha_m(\mu),\,m=1,\dots,k$, the free cumulants
of the~probability measure $\mu$. It is easy to see that $\alpha_1(\mu)=m_1(\mu),
\alpha_2(\mu)=m_2(\mu)-m_1^2(\mu),\,\alpha_3(\mu)=m_3(\mu)-3m_1(\mu)m_2(\mu)
+2m_1^3(\mu)$. In the~case $m_1(\mu)=0$ and $m_2(\mu)=1$ we have $\alpha_1(\mu)=0,\,
\alpha_2(\mu)=1,\,\alpha_3(\mu)=m_3(\mu)$ and $\alpha_4(\mu)=m_4(\mu)-2$.

If $\mu\in\mathcal M$ has moments of any order, 
that is $\beta_k(\mu)<\infty$ for any $k\in\mathbb N$, then there exist cumulants
$\alpha_m(\mu),\,m=1,\dots$, and we can consider the~formal power series
\begin{equation}\label{3.8}  
\phi_{\mu}(z)=\sum_{m=1}^{\infty}\frac{\alpha_{m}(\mu)}{z^{m-1}}.
\end{equation}
In addition $\phi_{\mu}(z)$ satisfies (\ref{3.7}) for any fixed $k\in\mathbb N$.
If $\mu$ has a~bounded support, $\phi_{\mu}(z)$ is an~analytic function
on the~domain $|z|>R$ with some $R>0$ and the~series (\ref{3.8}) converges
absolutely and uniformly for such $z$.

In order to express moments in terms of free cumulants introduce some notations.
A~partition $\bold V=\{V_1,V_2,\dots,V_s\}$ of the~set $[n]:=\{1,2,\dots,n\}$ consists of nonempty, pairwise disjoint 
blocks $V_1,V_2,\dots,V_s$ satisfying $\cup_{j=1}^s V_j=[n]$. $|V_j|$ denotes the~number of elements of $V_j$.
$\bold V$ is called a~non-crossing partition if for
$V_j=\{v_1,v_2,\dots,v_p\}$ and $V_l=\{w_1,w_2,\dots,w_q\}$ we have $w_m<v_1<w_{m+1}$ if and only if
$w_m<v_p<w_{m+1}\, (m=1,2,\dots,q-1)$. Denote by $Q_n^{(s)}$ the~number of non-crossing partitions of $[n]$ into $s$ blocks.
It is well-known, see~\cite{HiPe:2000}, that 
\begin{equation}\label{3.8a}
Q_n^{(s)}=\frac 1n {n\choose s}{n\choose s-1}. 
\end{equation}
The~following formula holds
\begin{equation}\label{3.8b}
m_n(\mu)=\sum_{\bold V\in NC(n)}\prod_{j=1}^s\alpha_{|V_j|}(\mu), 
\end{equation}
where the~summation here is over the~non-crossing partitions $\{V_1,V_2,\dots,V_s\}$.

Voiculescu~\cite{Vo:1993} showed for compactly supported probability measures that
there exist unique functions $Z_1, Z_2\in\mathcal F$ such that
$G_{\mu_1\boxplus\mu_2}(z)=G_{\mu_1}(Z_1(z))=G_{\mu_2}(Z_2(z))$ 
for all $z\in\mathbb C^+$.
Using Speicher's combinatorial approach~\cite{Sp:1998} to freeness,
Biane~\cite{Bi:1998} proved this result in the~general case.

 Bercovici and Belinschi~\cite{BeBel:2007},
Belinschi~\cite{Bel:2008},
Chistyakov and G\"otze \cite {ChG:2005} 
proved, using complex analytic methods, that
there exist unique functions $Z_1(z)$ and $Z_2(z)$ in the~class 
$\mathcal F$ such that, for $z\in\mathbb C^+$, 
\begin{equation}\label{3.9}
z=Z_1(z)+Z_2(z)-F_{\mu_1}(Z_1(z))\quad\text{and}\quad
F_{\mu_1}(Z_1(z))=F_{\mu_2}(Z_2(z)). 
\end{equation}
The~function $F_{\mu_1}(Z_1(z))$ belongs again to the~class $\mathcal F$ and 
there exists 
a~probability measure $\mu$ such that
$F_{\mu_1}(Z_1(z)) =F_{\mu}(z)$, where $F_{\mu}(z)=1/G_{\mu}(z)$ and 
$G_{\mu}(z)$ is the~Cauchy transform as in (\ref{2.2}). 

Specializing to $\mu_1=\mu_2=\dots=\mu_n=\mu$ write $\mu_1\boxplus\dots\mu_n=
\mu^{n\boxplus}$.
The~relation (\ref{3.9}) admits the~following
consequence (see for example \cite{ChG:2005}).

\begin{proposition}\label{3.3pro}
Let $\mu\in\mathcal M$. There exists a~unique function $Z_n\in\mathcal F$ 
such that
\begin{equation}\label{3.9*}
z=nZ_n(z)-(n-1)F_{\mu}(Z_n(z)),\quad z\in\mathbb C^+,
\end{equation}
and $F_{\mu^{n\boxplus}}(z)=F_{\mu}(Z_n(z))$.
\end{proposition}

Now we formulate and prove an important auxiliary result about the~behaviour of the function $Z_n(z)$ 
from the last proposition.

We obtain from (\ref{3.9*}) the~formula
\begin{equation}\label{3.9**}
Z_n^{(-1)}(z)=nz-(n-1)F_{\mu}(z) 
\end{equation}
for $z\in\Gamma_{\alpha,\beta}$ with some $\alpha,\beta>0$. By this formula we continue
the~function $Z_n^{(-1)}(z)$ as an~analytic function to $\mathbb C^+$. By (\ref{3.1}),
we have the~following representation for the~function $F_{\mu}(z)$
\begin{equation}\label{3.10}
F_{\mu}(z)=c+z+\int\limits_{\Bbb R}\frac{1+uz}{u-z}\,\sigma(du),\quad z\in\mathbb C^+,
\end{equation}
where $c\in\Bbb R$, and $\sigma$ is a~non negative finite
measure. Moreover, $c=\Re F_{\mu}(i)$ and $\sigma(\Bbb R)=\Im F_{\mu}(i)-1$.

Bercovici and Voiculescu~\cite{BeVo:1993}) proved the~next result.
\begin{proposition}~\label{3.4pro}
A~probability measure $\mu$ is $\boxplus$-infinitely
divisible if and only if the~function $\phi_{\mu}(z)$ has an~analytic
continuation to $\Bbb C^+$, with values in $\Bbb C^-\cup\Bbb R$, such that
\begin{equation}\label{3.10*}
\lim_{y\to +\infty}\frac{\phi_{\mu}(iy)}y=0.
\end{equation} 
\end{proposition}

It follows from Proposition~\ref{3.4pro} and (\ref{3.9}), (\ref{3.10}) that
a~probability measure $\nu_n$ such that $F_{\nu_n}(z)=Z_n(z),z\in\mathbb C^+$,
is $\boxplus$-infinitely divisible.

Consider the~functional equation, for every real fixed $x$,
\begin{equation}\label{3.11}
y\Big(1-(n-1)\int\limits_{\mathbb R}\frac{(1+u^2)\,\sigma(du)}{(u-x)^2+y^2}\Big)=0,\quad y>0.
\end{equation}
It is clear that this equation has at most one positive solution. If such solution exists, denote it
by $y_n(x)>0$ and assume that $y_n(x)=0$ otherwise. We note that the~curve $\gamma_n=\{y=y_n(x),\,x\in\mathbb R\}$
is Jordan one. 

Consider the~open domain $D_n:=\{z=x+iy,\,x,y\in\mathbb R: y>y_n(x)\}$.
\begin{lemma}\label{l3.3a}
The~map $Z_n(z):\mathbb C^+\mapsto D_n$ is univalent. Moreover the~function $Z_n(z),\,z\in\mathbb C^+$, 
is continuous up to the~real axis and the~real axis is mapped onto  the~curve $\gamma_n$.
\end{lemma}

\begin{lemma}\label{l3.3b} 
Let $c_1(\mu)>0$ and let $Z_n(z)$ be the solution of the equation $(\ref{3.9*})$. Then the following lower bound holds
$$
|Z_n(z)|\ge \frac 14\sqrt{c_1(\mu)(n-1)},\quad z\in\mathbb C^+,\quad n\ge c(\mu).
$$
\end{lemma}
The lemmas~\ref{l3.3a} and \ref{l3.3b} were proved in \cite{ChG:2013}.

We need the~following Bercovici-Voiculescu result~\cite{BeVo:1993}.
\begin{proposition}\label{3.3b}
If $\mu,\mu',\nu$, and $\nu'$ are probability measures, then
\begin{equation}\notag
\Delta(\mu\boxplus\nu,\mu'\boxplus\nu')\le\Delta(\mu,\mu')+\Delta(\nu,\nu').
\end{equation}
\end{proposition}

In the~sequel we use the~following well-known inequality (see \cite{La:1966}).
\begin{proposition}\label{3.3c}
Let $Q(x),\,x\in\mathbb R$ be a~distribution function. Then
\begin{equation}\notag
\int\limits_{\{|x|>2/u\}} dQ(x)\le\frac 2u \int\limits_{[0,u]}(1-\Re\varphi(t;Q))\,dt,
\quad u>0.
\end{equation} 
\end{proposition}


\begin{lemma}\label{l3.5}
If a~probability measure $\mu$ has a finite second moment
$m_2(\mu)$,
then the~solution of the equation $(\ref{3.9*})$ $Z_n(z)$ 
admits the~representation
\begin{equation}\label{3.15}
Z_n(z)=z-(n-1)m_1(\mu)+\int\limits_{\mathbb R}\frac {\tau(du)}{u-z},
\quad z\in\Bbb C^+,
\end{equation}
whith a~non negative measure $\tau$ such that  $\tau(\mathbb R)=(n-1)(m_2(\mu)-m_1^2(\mu))$.
\end{lemma}
\begin{proof}
Since $F_{\mu}(z):=1/G_{\mu}(z)\in\Cal N$ and $F_{\mu}(iy)/(iy)\to 1$ as 
$y\to\infty$, we have from Nevanlinna's integral representation
$F_{\mu}(z)=z-m_1(\mu)+f(z)$, where $f\in\Cal N$ and $f(iy)/y\to 0$ as 
$y\to\infty$. In view of the relation 
$$
\lim_{y\to\infty}(iy)^{3}\Big(G_{\mu}(iy)-\frac 1{iy}-
\frac{m_1(\mu)}{(iy)^2}\Big)
=m_{2}(\mu),
$$ 
we have, for $y\to\infty$,
\begin{equation}\label{3.16}
F_{\mu}(iy)-iy=-m_1(\mu)-\frac {m_2(\mu)-m_1^2(\mu)}{iy}+o\Big(\frac 1{y}\Big).
\end{equation}
From here we conclude
that
$$
f(z)=\int\limits_{\mathbb R}\frac{\sigma(du)}{u-z},
$$ 
whith the~measure $\sigma$ such that
$\sigma(\mathbb R)=m_2(\mu)-m_1^2(\mu)$. 

The~solution $Z_n(z)$ of (\ref{3.9*}) belongs to the~class~ 
$\Cal F$, therefore $Z_n(z)=z+g(z),\,z\in\Bbb C^+$, where $g\in\mathcal N$
and $g(iy)/y\to 0$ as $y\to\infty$. Rewrite (\ref{3.9*}) in the~form
\begin{equation}\label{3.17}
\frac 1{n-1}\big(Z_n(z)-z\big)=-m_1(\mu)+\int\limits_{\Bbb R}\frac{\sigma(du)}{u-Z_n(z)}.
\end{equation}
Since $\sup_{y\ge 1}|yf(iy)|<\infty$, we see that
\begin{equation}\label{3.18}
Z_n(z)-z=-(n-1)m_1(\mu)+\int\limits_{\mathbb R}\frac{\tau\,(du)}{u-z},\quad z\in\mathbb C^+,
\end{equation}
where $\tau$ is a~finite non negative measure. In addition, by
the~relations $-\lim_{y\to\infty} iyf(iy)$ $\sigma(\mathbb R)=m_2(\mu)-m_1^2(\mu)$ and
$\lim_{y\to\infty}(iy)(Z_n(iy)-iy+(n-1)m_1(\mu))=\tau(\mathbb R)$, we conclude that
$\tau(\mathbb R)=(n-1)(m_2(\mu)-m_1^2(\mu))$. The lemma is proved.

\end
{proof}
The next lemma is a refinement of Lemma~\ref{l3.5}. We save all notation of this lemma.

\begin{lemma}\label{l3.6}
Let a probability measure $\mu$ has compact support, i.e. $\mu(\mathbb R)=\mu([-R,R])$ with some $R>0$. Then the solution of $(\ref{3.9*})$ $Z_n(z)$ admits
the representation $(\ref{3.15})$, where
$\tau(\mathbb R)=\tau([-R_1,R_1])$
with $R_1=4(n-1)(R+1)(\sigma(\mathbb R)+1)$.
\end{lemma}
\begin{proof}
The function $F_{\mu}(z)$ is analytic in $\mathbb C \setminus [-R,R]$ and is real-valued for $\mathbb R\setminus[-R,R]$. Therefore $\sigma$ from (\ref{3.10})
has compact support and $\sigma(\mathbb R)=\sigma([-R,R])$.

Denote by $\mu_{Z_n}$ the probability measure with the Cauchy transform 
$\frac{1}{Z_n(z)}$.
By (\ref{3.17}), we see that
$$
\alpha_1(\mu_{Z_n})=
(n-1)m_1(\mu),\quad
\alpha_k(\mu_{Z_n})=(n-1)
s_{k-1}(\sigma),\quad k=2,3,\dots,
$$
where $s_{k-1}(\sigma)=\int_{\mathbb R}u^{k-1}\,\sigma(du)$.
Therefore we have the bounds
\begin{equation}\notag
 |\alpha_k(\mu_{Z_n})|
 \le (n-1)R^{k-1}\sigma(\mathbb R)\le
 ((n-1)(R+1)(\sigma(\mathbb R)+1))^k
\end{equation}
for $k=2,\dots$. Applying
these estimates to (\ref{3.8b}), we obtain
\begin{equation}\notag
 |m_k(\mu_{Z_n})|
 \le (4(n-1)(R+1)(\sigma(\mathbb R)+1))^k
\end{equation}
for $k=1,2,\dots$. See \cite{NSp:2006}, p.p. 218--219, as well.

The function $G_{\mu_{Z_n}}(z)$ is analytic in the domain $|z=x+iy|>R_1$, is not equal to zero and real-valued for $|x|>R_1$. Therefore $Z_n(z)$ is analytic for $|z|>R_1$ and real-valued for $|x|>R_1$
as well. Therefore  we obtain the assertion of the lemma from (\ref{3.18}) and from the inversion formula (\ref{3.4}).
\end{proof}

\section{ Approximations in Limit Theorems}
{\bf Proof of Theorem~\ref{th4}.}
In the~sequel we assume without loss of generality $m_1(\mu)=0$ and denote $\mu_n:=\mu^{n\boxplus}$ and $\nu_n:=\rho_n^{n\boxplus}$.
By the~assumptions of the~theorem, there exists a~subsequence of increasing
positive integers $\{n_k\}_{k=1}^{\infty}$ such that
\begin{equation}\label{4.1}
\Delta(\mu_{n_k},\nu_{n_k})\le \varepsilon_{n_k}:=e^{-\beta' \sqrt{n_k}},
\end{equation}
where 
$0<\beta'<\beta$. 
By the~assumptions of the~theorem, $\beta_d(\mu)<\infty,\,d\ge 3$, therefore, in view of  the~well-known Minkowski inequality 
$(\beta_d(\mu^{n\boxplus}))^{1/d}\le n(\beta_d(\mu))^{1/d}$
(see \cite{Ta:2003}), 
we have 
\begin{equation}\label{4.2}
\mu_n(\mathbb R\setminus[-t,t])\le n^d\beta_d(\mu)t^{-d}\quad\text{for all}\quad t>0.
\end{equation}

By~(\ref{3.7a}), $G_{\mu}(z)=\frac 1z+\frac{m_2(\mu)}{z^3}+o\Big(\frac 1{z^3}\Big)$,
when $z\to\infty$ in the~angle $\delta\le\arg z\le\pi-\delta$ for any small fixed $\delta>0$.
Therefore $F_{\mu}(z)=1/G_{\mu}(z)=z-\frac{m_2(\mu)}{z}+o\Big(\frac 1{z}\Big)$ for the~same $z$.
Hence, by (\ref{3.3}),
\begin{equation}\label{4.3}
F_{\mu}(z)=z+\int\limits_{\mathbb R}\frac{\sigma(du)}{u-z},\quad z\in\Bbb C^+,
\end{equation}
where $\sigma$ is a~non negative finite measure such that $\sigma(\mathbb R)=m_2(\mu)$.

By Proposition~\ref{3.3pro}, there exists a~unique function $Z_n(z)\in\mathcal F$ 
such that
\begin{equation}\label{4.4}
z=nZ_n(z)-(n-1)F_{\mu}(Z_n(z)),\quad z\in\Bbb C^+,
\end{equation}
and $G_{\mu_{n}}(z)=G_{\mu}(Z_n(z)),\,z\in\Bbb C^+$. By Lemma~\ref{l3.5},
we conclude that $Z_n(z)$ admits the~representation
\begin{equation}\label{4.5}
Z_n(z)=z+\int\limits_{\mathbb R}\frac{\sigma_n(du)}{u-z},\quad z\in\Bbb C^+,
\end{equation}
where $\sigma_n$ is a~non negative finite measure such that $\sigma_n(\mathbb R)=(n-1)m_2(\mu)$. 


By Proposition~\ref{3.3pro}, there exists a~unique function $W_n(z)\in\mathcal F$ 
such that
\begin{equation}\label{4.6}
z=nW_n(z)-(n-1)F_{\rho_n}(W_n(z)),\quad z\in\Bbb C^+,
\end{equation}
and $G_{\nu_n}(z)=G_{\rho_n}(W_n(z)),\,z\in\Bbb C^+$. 

Now we evaluate the~closeness of the~functions $Z_n(z)$ and $W_n(z)$.
\begin{lemma}\label{l4.1}
The~following upper bound holds
\begin{equation}\label{4.7}
|Z_n(z)-W_n(z)|\le \frac{18\pi\varepsilon_n}{\Im z}\Big(|z|+\frac{nm_2(\mu)}{\Im z}\Big)^2
\end{equation} 
for $|\Re z |\le R_n$ and $\Im z\ge 1/R_n$, where $R_n:=(10\sqrt{\varepsilon_n nm_2(\mu)})^{-1}$. 
\end{lemma}
\begin{proof}
Integrating by parts we obtain the~formula
\begin{equation}\notag
r_n(z):=G_{\mu_{n}}(z)- G_{\nu_n}(z)=-\int\limits_{\mathbb R}(\mu_{n}((-\infty,u))-\nu_n((-\infty,u)))
\frac{du}{(z-u)^2},\quad z\in\mathbb C^+.
\end{equation}
Using (\ref{4.1}), we easily deduce from this formula the~estimate 
\begin{equation}\label{4.8}
|r_n(z)|\le \varepsilon_n\int\limits_{\mathbb R}\frac{du}{|z-u|^2}=\frac{\pi\varepsilon_n}{\Im z},
\quad z\in\mathbb C^+.
\end{equation}
By (\ref{4.4}) and (\ref{4.6}), we have
\begin{equation}\notag
r_n(z)=n(n-1)\frac{W_n(z)-Z_n(z)}{(nZ_n(z)-z)(nW_n(z)-z)} 
\end{equation}
and hence
\begin{equation}\label{4.9}
W_n(z)-Z_n(z)=\frac{r_n(z)}{n(n-1)}\frac{(nZ_n(z)-z)^2}{1-r_n(z)(nZ_n(z)-z)/(n-1)}.
\end{equation}
By (\ref{4.5}), we see that
\begin{equation}\label{4.10}
|Z_n(z)|\le |z|+(n-1)m_2(\mu)\frac 1{\Im z},\quad z\in\mathbb C^+.
\end{equation}
By (\ref{4.8}) and (\ref{4.10}), we conclude that
\begin{equation}\label{4.11}
\frac {|r_n(z)|}{n-1}(n|Z_n(z)|+|z|)\le\frac{\pi nm_2(\mu)\varepsilon_n}{(\Im z)^2}
+\frac{3\pi\varepsilon_n |z|}{\Im z}<\frac 1{5}
\end{equation}
for $|\Re z |\le R_n$ and $\Im z\ge 1/R_n$.
Using (\ref{4.10}) and (\ref{4.11}) we deduce from (\ref{4.9})
\begin{align}
&|W_n(z)-Z_n(z)|\le \frac{|r_n(z)|}{n(n-1)}\frac{(n|Z_n(z)|+|z|)^2}{|1-|r_n(z)|(n|Z_n(z)|+|z|)/(n-1)|}
\notag\\
&\le 2 \frac{\pi\varepsilon_n}{\Im z}\Big(3|z|+\frac{nm_2(\mu)}{\Im z}\Big)^2\notag
\end{align}
for $|\Re z |\le R_n$ and $\Im z\ge 1/R_n$.
The~lemma is proved.
\end{proof}

In the~sequel we need an~information about Nevalinna functions $\frac 1{u-Z_n(z)},\,z\in\mathbb C^+$ 
for all $u\in\mathbb R$.
\begin{lemma}\label{l4.2}
For every $u\in\mathbb R$, Nevalinna functions $\frac 1{u-Z_n(z)},\,z\in\mathbb C^+$, admit
the~integral representation
\begin{equation}\label{4.12}
\frac 1{u-Z_n(z)}=\int\limits_{\mathbb R}\frac{\zeta_n(u;ds)}{s-z},
\end{equation}
where $\zeta_n(u;ds)$ are probability measures such that 
\begin{equation}\notag
 \zeta_n(u;[u-\sqrt{2\sigma_n(\mathbb R)},u+\sqrt{2\sigma_n(\mathbb R)}])\ge\frac 12
\end{equation}
for all $u\in\mathbb R$ and $n\in\mathbb N$.
\end{lemma}
\begin{proof}
By (\ref{4.3}), we have the~relation $Z_n(z)=z-\sigma_n(\mathbb R)/z+o(1/z)$ as $z\to \infty$ and $z\in\Gamma_{\alpha,\beta}$ for some
positive $\alpha$ and $\beta$. Using it we easily deduce, for $z\to \infty$ and $z\in\Gamma_{\alpha,\beta}$,
\begin{equation}\notag
\frac 1{u-Z_n(z)}=-\frac 1z-\frac u{z^2}-\frac{\sigma_n(\mathbb R)+u^2}{z^3}+\frac{o(1)}{z^3}. 
\end{equation}
By Lemma~\ref{3.7a}, we conclude that $m_1(\zeta_n(u;ds))=u$ and $m_2(\zeta_n(u;ds))=u^2+\sigma_n(\mathbb R)$.
Then, by Chebyshev's inequality, we have the~bound 
\begin{equation}\notag
\zeta_n(u;\mathbb R\setminus[u-\sqrt{2\sigma_n(\mathbb R)},u+\sqrt{2\sigma_n(\mathbb R)}])\le
\frac{\sigma_n(\mathbb R)}{2\sigma_n(\mathbb R)}=\frac 12.
\end{equation}
The~lemma is proved.
\end{proof}

Our next step is to prove that the~tails of the~measure $\rho_n$ are small.
Denote
\begin{equation}\label{4.12*}
\varepsilon_n^*:=\varepsilon_n+\big(\frac n{N_n}\Big)^3+\frac{\eta_n}{N_n}
+\varepsilon_n\Big(\frac {N_n}{\eta_n^3}+\frac{n^2}{\eta_n^5}\Big), 
\end{equation}
where $N_n=\varepsilon_n^{-1/13}$ and 
$\eta_n=1/N_n^2$.

\begin{lemma}\label{l4.3}
The following upper bounds hold 
\begin{equation}\notag
\rho_n((-\infty,-2N_n))+\rho_n((2N_n,\infty))\le c(\mu)\varepsilon_n^*. 
\end{equation}
\end{lemma}
\begin{proof}
Using (\ref{4.1}) and (\ref{4.2}), we get the~inequality
\begin{align}\label{4.13}
&-\int\limits_{[N_n,\infty)}\Im G_{\nu_n}(x+i\eta_n)\,dx=\int\limits_{\mathbb R}\nu_n(ds)
\int\limits_{[N_n,\infty)}\frac{\eta_n\,dx}{(x-s)^2+\eta_n^2}\notag\\
&=\Big(\int\limits_{[-N_n/2,N_n/2]}+\int\limits_{\mathbb R\setminus[-N_n/2,N_n/2]}\Big)\nu_n(ds)
\int\limits_{[(N_n-s)/\eta_n,\infty)}\frac{dt}{t^2+1}\notag\\
&\le \pi\nu_n(\mathbb R\setminus[-N_n/2,N_n/2])+\frac{2\eta_n}{N_n}
\le 2\pi\varepsilon_n+\frac{8\pi
\beta_3(\mu)n^3}{N_n^{3}}+
\frac{2\eta_n}{N_n}.
\end{align}

On the~other hand we have the~relation
\begin{align}\label{4.14}
-\int\limits_{[N_n,\infty)}\Im G_{\nu_n}(x+i\eta_n)\,dx
&=\int\limits_{[N_n,\infty)}dx\int\limits_{\mathbb R}\Im\frac 1{u-W_n(x+i\eta_n)}
\,\rho_n(du)\notag\\
&=\int\limits_{\mathbb R}\rho_n(du)\int\limits_{[N_n,\infty)}
\Im\frac 1{u-W_n(x+i\eta_n)}\,dx.
\end{align}
In view of Lemma~\ref{l4.1}, we conclude
\begin{align}
\Big|\Im\Big(\frac 1{u-Z_n(z)}-\frac 1{u-W_n(z)}\Big)\Big|&= \Big|\Im\frac{W_n(z)-Z_n(z)}
{(u-Z_n(z))(u-W_n(z))}\Big|\le\frac{|W_n(z)-Z_n(z)|}{(\Im z)^2}\notag\\
&\le \frac{18\pi\varepsilon_n}{(\Im z)^3}\Big(|z|+\frac{nm_2(\mu)}{\Im z}\Big)^2
\label{4.15}
\end{align}
for $|\Re z|\le R_n$ and $\Im z\ge 1/R_n$. Note that $R_n\ge N_n$ for sufficiently large $n$.
Applying (\ref{4.15}) to (\ref{4.14}) and recalling (\ref{4.12}), we obtain from (\ref{4.13}) the~lower bound
\begin{equation}\label{4.16}
c(\mu)\varepsilon_n^*\ge c(\mu)\varepsilon_n\Big(\frac{N_n}{\eta_n^3}+\frac{n^2}{\eta_n^5}\Big)
\ge\int\limits_{\mathbb R}\rho_n(du)\int\limits_{[N_n,\infty} \Im\frac 1{u-Z_n(x+i\eta_n)}\,dx.
\end{equation}
By Lemma~\ref{l4.2}, we deduce the~inequality , for $u\ge 2N_n$,
\begin{align}
&\int\limits_{[N_n,\infty)} \Im\frac 1{u-Z_n(x+i\eta_n)}\,dx
=\int\limits_{[N_n,\infty)}\int\limits_{\mathbb R}\frac{\eta_n\,\zeta(u;ds)}{(s-x)^2+\eta_n^2}\,dx\notag\\
&=\int\limits_{\mathbb R}\int\limits_{(N_n-s)/\eta_n}^{\infty}\frac{dv}{v^2+1}\,\zeta(u;ds)
\ge\int\limits_{[u-\sqrt{2\sigma_n(\mathbb R)},u+\sqrt{2\sigma_n(\mathbb R)}]}\int\limits_{-N_n/(2\eta_n)}^{N_n/(2\eta_n)}
\frac{dv}{v^2+1}\,\zeta(u;ds)\notag\\
&\ge \frac {\pi}2\int\limits_{[u-\sqrt{2\sigma_n(\mathbb R)},u+\sqrt{2\sigma_n(\mathbb R)}]}\zeta(u;ds)\ge\frac{\pi}4.\label{4.17}
\end{align}
From (\ref{4.16}) and (\ref{4.17}) it follows that
$\int\limits_{[2N_n,\infty)}\rho_n(du)\le c(\mu)\varepsilon_n^*$.
In the~same way we obtain the~estimate 
$\int\limits_{(-\infty,-2N_n]}\rho_n(du)\le c(\mu)\varepsilon_n^*$
as well. Thus, the lemma is proved.
\end{proof}

Let $\nu\in\mathcal M$. In the~sequel we denote by $\tilde{\nu}$ a~probability measure such that $\tilde{\nu}(S)=\nu(S)$ for 
all Borel sets $S\subseteq [-2N_n,2N_n]\setminus\{0\}$ and
$\tilde{\nu}(\{0\})=\nu(\{0\})+\nu(\mathbb R\setminus[-2N_n,2N_n])$.

Now in place of the~measures $\mu$ and $\rho_n$ we will consider probability measures $\tilde{\mu}$ and $\tilde{\rho}_{n}$.
By the~inequality $\mu(\mathbb R\setminus[-2N_n,2N_n])\le \beta_d(\mu)/(2N_n)^d$ and Lemma~\ref{l4.3}, 
$\Delta(\tilde{\mu},\mu)\le \beta_d(\mu)/(2N_n)^d$ and 
$\Delta(\tilde{\rho}_{n},\rho_n)\le c(\mu)\varepsilon_n^*$. By (\ref{4.1}) and
Proposition~\ref{3.3b}, we have the~following upper bound
\begin{equation}\label{4.21}
\Delta(\tilde{\mu}_{n},\tilde{\nu}_{n})
\le \Delta(\tilde{\mu}_n,\mu_n)+\Delta(\mu_n,\nu_n)+\Delta(\nu_n,\tilde{\nu}_n)
\le c(\mu)n\varepsilon_n^*,
\end{equation}
where $\tilde{\mu}_{n}:=\tilde{\mu}^{n\boxplus}$ and $\tilde{\nu}_{n}:=\tilde{\rho}_{n}^{n\boxplus}$. Since $\tilde{\mu}([-2N_n,2N_n])=
\tilde{\rho}_{n}([-2N_n,2N_n])=1$, we conclude that $\tilde{\mu}_{n}([-2nN_n,2nN_n])=\tilde{\nu}_{n}([-2nN_n,2nN_n])=1$.

The~function $F_{\tilde{\mu}}(z)$ admits the~representation
\begin{equation}\label{4.3*}
F_{\tilde{\mu}}(z)=z+m_1(\tilde{\mu})+\int\limits_{\mathbb R}\frac{\tilde{\sigma}(du)}{u-z},\quad z\in\Bbb C^+,
\end{equation}
where $\tilde{\sigma}$ is a~non negative finite measure such that $\tilde{\sigma}(\mathbb R)=m_2(\tilde{\mu})-m_1^2(\tilde{\mu})$.
This representation we obtain in the~way as we have got representation (\ref{4.3}).

By Proposition~\ref{3.3pro}, there exists a~unique function $\tilde{Z}_n(z)\in\mathcal F$ 
such that
\begin{equation}\label{4.4*}
z=n\tilde{Z}_n(z)-(n-1)F_{\tilde{\mu}}(\tilde{Z}_n(z)),\quad z\in\Bbb C^+,
\end{equation}
and $G_{\tilde{\mu}_{n}}(z)=G_{\tilde{\mu}}(\tilde{Z}_n(z)),\,z\in\Bbb C^+$. Note that, by Lemma~\ref{l3.5},
we conclude that $\tilde{Z}_n(z)$ admits the~representation
\begin{equation}\label{4.5*}
\tilde{Z}_n(z)=z-(n-1)m_1(\tilde{\mu})+\int\limits_{\mathbb R}\frac{\tilde{\sigma}_{n}(du)}{u-z},\quad z\in\Bbb C^+,
\end{equation}
where $\tilde{\sigma}_{n}$ is a~non negative finite measure such that $\tilde{\sigma}_{n}(\mathbb R)=(n-1)(m_2(\tilde{\mu})-m_1^2(\tilde{\mu}))$. 

In the~same way, by Proposition~\ref{3.3pro} there exists a~unique function $\tilde{W}_n(z)\in\mathcal F$ 
such that
\begin{equation}\label{4.22}
z=n\tilde{W}_{n}(z)-(n-1)F_{\tilde{\rho}_{n}}(\tilde{W}_{n}(z)),\quad z\in\Bbb C^+,
\end{equation}
and $G_{\tilde{\nu}_{n}}(z)=G_{\tilde{\rho}_{n}}(\tilde{W}_{n}(z)),\,z\in\Bbb C^+$. 
By Lemma~\ref{l3.5}, $\tilde{W}_{n}(z)$ admits the~representation
\begin{equation}\label{4.24}
\tilde{W}_{n}(z)=z-(n-1)m_1(\tilde{\rho}_{n})+\int\limits_{\mathbb R}\frac{\tilde{\tau}_{n}du)}{u-z},\quad z\in\Bbb C^+,
\end{equation}
where $\tilde{\tau}_{n}$ is a~non-negative finite measure such that $\tilde{\tau}_{n}(\mathbb R)=(n-1)(m_2(\tilde{\rho}_{n})-m_1^2(\tilde{\rho}_{n}))$.


Now we need to estimate closeness of the~functions $\tilde{Z}_n(z)$ and $\tilde{W}_{n}(z)$.
\begin{lemma}\label{l4.4}
The following bound holds
\begin{equation}\notag
|\tilde{Z}_n(z)-\tilde{W}_{n}(z)|\le c(\mu)n^3N_n^2\varepsilon_n^*\quad\text{for}\quad 1\le \Im z\le 6nN_n.
\end{equation} 
\end{lemma}
\begin{proof}
We prove this lemma using the~arguments of the~proof of Lemma~\ref{l4.1}. 
Integrating by parts and using (\ref{4.21}) we obtain the~following estimate 
for $\tilde{r}_n(z):=G_{\tilde{\mu}_{n}}(z)- G_{\tilde{\nu}_{n}}(z)$
\begin{equation}\label{4.26}
|\tilde{r}_n(z)|\le c(\mu)n\varepsilon_n^*\int\limits_{[-2nN_n,2nN_n]}\frac{du}{|z-u|^2}\le
\begin{cases}
\frac{c(\mu) n\varepsilon_n^*}{(\Im z)^2}, &if \,\, |\Re z|\le 6nN_n,\,\,\Im z>0\cr
\frac{c(\mu)n^2N_n\varepsilon_n^*}{|z|^2}, &if \,\, |\Re z|> 6nN_n,\,\,\Im z>0.
\end{cases}
\end{equation}
By (\ref{4.4*}) and (\ref{4.22}), we have
\begin{equation}\notag
\tilde{r}_n(z)=n(n-1)\frac{\tilde{W}_n(z)-\tilde{Z}_n(z)}{(n\tilde{Z}_n(z)-z)(n\tilde{W}_n(z)-z)} 
\end{equation}
and hence
\begin{equation}\label{4.27}
\tilde{W}_n(z)-\tilde{Z}_n(z)=\frac{\tilde{r}_n(z)}{n(n-1)}\frac{(n\tilde{Z}_n(z)-z)^2}{1-\tilde{r}_n(z)(n\tilde{Z}_n(z)-z)/(n-1)}.
\end{equation}

Repeating the~arguments which we used in the~proof of Lemma~\ref{l4.1} one can easily obtain
the~following estimate, for $|\Re z|\le 6nN_n,\,\Im z\ge 1$,
\begin{equation}\label{4.27*}
|\tilde{Z}_n(z)-\tilde{W}_n(z)|\le c(\mu)n\varepsilon_n^*(|z|+n)^2. 
\end{equation}

From (\ref{4.5*}) it follows the~upper bound 
\begin{equation}\label{4.27**}
|\tilde{Z}_n(z)|\le |z|+(n-1)(|m_1(\tilde{\mu})|+m_2(\tilde{\mu})),\quad\Im z\ge 1. 
\end{equation} 
By (\ref{4.26}) and (\ref{4.27**}), we conclude that, for $|\Re z|>6nN_n,\,\Im z\ge 1$,
\begin{align}\label{4.28}
|\tilde{r}_n(z)|\Big(2|\tilde{Z}_n(z)|+\frac{|z|}{n-1}\Big)&\le c(\mu)\frac{n^2N_n\varepsilon_n^*}{|z|^2}
\big(|z|+(n-1)(|m_1(\tilde{\mu})|+m_2(\tilde{\mu}))\big)\notag\\
&\le \frac{c(\mu)n^2N_n\varepsilon_n^*}{|z|}<\frac 12,\quad n\ge c(\mu).
\end{align}
Using (\ref{4.26}), (\ref{4.27**}) and (\ref{4.28}) we obtain from (\ref{4.27}) the assertion of the lemma for $1\le\Im z\le 6nN_n$ and $|\Re z|> 6nN_n$. Recalling (\ref{4.27*}), we arrive at the assertion of the lemma.\end{proof}

Denote $S_n(z):=G_{\tilde{\mu}}(\tilde{Z}_n(z))-G_{\tilde{\rho}_n}(\tilde{W}_{n}(z))$ for $z\in\mathbb C^+$ and research
the~behavior of this function. We represent this function as follows
\begin{equation}\label{4.29a}
S_n(z)=T_n(z)+R_n(z):=\int\limits_{[-2N_n,2N_n]}\frac{(\tilde{\mu}-\tilde{\rho}_{n})(du)}{\tilde{Z}_n(z)-u}+
\int\limits_{[-2N_n,2N_n]}\frac {(\tilde{W}_{n}(z)-\tilde{Z}_n(z))\,\tilde{\rho}_{n}(du)}{(\tilde{Z}_n(z)-u)(\tilde{W}_{n}(z)-u)} 
\end{equation}
for $z\in\mathbb C^+$.

Here we need to estimate the~function $R_n(z)$. This is done in the~next lemma.
\begin{lemma}\label{l4.5}
The~following upper bound holds
\begin{equation}\notag
|R_n(z)|\le c(\mu)n^3N_n^2\varepsilon_n^*\int\limits_{[-2N_n,2N_n]}\frac{\tilde{\rho}_{n}(du)}{|\tilde{Z}_n(z)-u|^2},
\quad 1\le\Im z\le 6nN_n. 
\end{equation}
\end{lemma}
\begin{proof}
We note, using Lemma~\ref{l4.4} and (\ref{3.5**}), that, for $|u|\le 2N_n$ and
$\Im z\ge 1$,
\begin{align}\label{4.29}
|\tilde{W}_{n}(z)-u|&\ge |\tilde{Z}_n(z)-u|-|\tilde{W}_{n}(z)-\tilde{Z}_n(z)|\notag\\
&\ge\frac 12|\tilde{Z}_n(z)-u|+\frac 12\Im z
-c(\mu)n^3N_n^2\varepsilon_n^*\ge\frac 12|\tilde{Z}_n(z)-u|.
\end{align}
From (\ref{4.29}) and Lemma~\ref{l4.4} it follows, for $|u|\le 2N_n$ and $1\le\Im z\le 6nN_n$, 
\begin{equation}\notag
\Big|\frac{\tilde{W}_{n}(z)-\tilde{Z}_n(z)}
{(\tilde{Z}_n(z)-u)(\tilde{W}_{n}(z)-u)}\Big|\le\frac{c(\mu)n^3N_n^2\varepsilon_n^*}{|\tilde{Z}_n(z)-u|^2}.
\end{equation}
Using this estimate we arrive immediately at the~assertion of the~lemma.
\end{proof}

In order to complete the~proof of the~theorem we need an~additional information about
the~function $\tilde{Z}_n(z)$.

We obtain from (\ref{4.4*}) the~following expression for
$\tilde{Z}_n^{(-1)}(z)$
\begin{equation}\label{4.32}
\tilde{Z}_n^{(-1)}(z)=nz-(n-1)F_{\tilde{\mu}}(z) 
\end{equation}
for $z\in\Gamma_{\alpha,\beta}$ with some $\alpha,\beta>0$. By this formula, we continue
the~function $\tilde{Z}_n^{(-1)}(z)$ as an~analytic function in $\mathbb C^+$. 

Using the representaion (\ref{4.3*})
we consider the~functional equation, for every fixed $x\in\mathbb R$,
\begin{equation}\label{4.33}
y\Big(1-(n-1)\int\limits_{[-2N_n,2N_n]}\frac{\tilde{\sigma}(du)}{(u-x)^2+y^2}\Big)=0,\quad y>0.
\end{equation}
This equation has at most one positive solution. Denote it by $\tilde{y}_n(x)>0$ if 
such a solution exists and by $\tilde{y}_n(x)=0$ otherwise. We see that the~curve $\tilde{\gamma}_n$ given by the~equation
$y=\tilde{y}_n(x),\,x\in\mathbb R$, 
is a Jordan curve and, as it is easy to see, that $\tilde{y}_n(x)=0$ for $|x|> 2N_n+\sqrt{\tilde{\sigma}(\mathbb R)(n-1)}$
and $\tilde{y}_n(x)<\sqrt{\tilde{\sigma}(\mathbb R)(n-1)}$ for $x\in\mathbb R$. 

Consider the~open domain $\tilde{D}_n:=\{z=x+iy,\,x,y\in\mathbb R: y>\tilde{y}_n(x)\}$.
By Lemma~\ref{l3.3a}, the~map $\tilde{Z}_n(z):\mathbb C^+\mapsto \tilde{D}_n$ is univalent. Moreover the~function $\tilde{Z}_n(z),\,z\in\mathbb C^+$, 
is continuous up to the~real axis and the~real axis passages on the~curve $\tilde{\gamma}_n$. 

 

Let the~curve $\gamma_n$ be given by the~equation $w=\tilde{Z}_n^{(-1)}(x+ih_n),x\in\mathbb R$,
where $h_n:=10\sqrt{\tilde{\sigma}(\mathbb R)n}$.
By (\ref{4.3*}) and (\ref{4.32}), we easily get, for all real $x$,
\begin{equation}\label{4.34}
\frac 12 h_n\le \Im \tilde{Z}_n^{(-1)}(x+ih_n)=h_n
\Big(1-(n-1)\int\limits_{[-2N_n,2N_n]}\frac{\tilde{\sigma}(du)}{(x-u)^2+h_n^2}\Big)\le h_n
\end{equation}
and, for $|x|\ge 2N_n+2h_n$,
\begin{align}\label{4.35}
&|\Re \tilde{Z}_n^{(-1)}(x+ih_n)|=\Big|x-m_1(\tilde{\mu})-(n-1)
\int\limits_{[-2N_n,2N_n]}\frac{(x-u)\,\tilde{\sigma}(du)}{(x-u)^2+h_n^2}\Big|\ge\frac 12|x|.
\end{align}

Return to formula (\ref{4.29a}) and write with its help the~relation
\begin{equation}\label{4.36}
\int\limits_{[-2nN_n,2nN_n]}\frac{(\tilde{\mu}_{n}-\tilde{\nu}_{n})(du)}{\tilde{Z}_n^{(-1)}(x+ih_n)-u}=
\int\limits_{[-2N_n,2N_n]}\frac{(\tilde{\mu}-\tilde{\rho}_{n})(du)}{x-u+ih_n} +R_n(\tilde{Z}_n^{(-1)}(x+ih_n))
\end{equation}
for $x\in\mathbb R$. By Lemma~\ref{l4.5}, we have the~estimate
\begin{equation}\label{4.37}
|R_n(\tilde{Z}_n^{(-1)}(x+ih_n))|\le c(\mu)n^3N_n^2\varepsilon_n^*\int\limits_{[-2N_n,2N_n]}\frac{\tilde{\rho}_{n}(du)}
{(x-u)^2+h_n^2},\quad x\in\mathbb R. 
\end{equation}
From (\ref{4.37}) it follows
\begin{align}\label{4.38}
&\Big|\int\limits_{\mathbb R}e^{itx}\Im R_n((\tilde{Z}_n)^{(-1)}(x+ih_n))\,dx\Big|
\le \int\limits_{\mathbb R}|R_n((\tilde{Z}_n)^{(-1)}(x+ih_n))|\,dx\notag\\
&\le c(\mu)n^3N_n^2\varepsilon_n^*\int\limits_{\mathbb R}dx\int\limits_{[-2N_n,2N_n]}\frac{\tilde{\rho}_{n}(du)}
{(x-u)^2+h_n^2}\le c(\mu)n^3N_n^2\varepsilon_n^*h_n^{-1},\quad t\in\mathbb R. 
\end{align}

Denote $V_n(u):=\tilde{\mu}_{n}((-\infty,u))-\tilde{\nu}_{n}((-\infty,u)),\,u\in\mathbb R$.
Integrating by parts, we obtain
\begin{align}
&\Big|\Im \int\limits_{[-2nN_n,2nN_n]}\frac{(\tilde{\mu}_{n}-\tilde{\nu}_{n})(du)}{\tilde{Z}_n^{(-1)}(x+ih_n)-u}\Big|\notag\\
&=\Big|\int\limits_{[-2nN_n,2nN_n]}\frac{\Im \tilde{Z}_n^{(-1)}(x+ih_n)(\tilde{\mu}_{n}-\tilde{\nu}_{n})(du)}
{(\Re \tilde{Z}_n^{(-1)}(x+ih_n)-u)^2+(\Im (\tilde{Z}_n)^{(-1)}(x+ih_n))^2}\Big|\notag\\
&\le 2\int\limits_{[-2nN_n,2nN_n]}\frac{\Im \tilde{Z}_n^{(-1)}(x+ih_n)|\Re \tilde{Z}_n^{(-1)}(x+ih_n)-u|
|V_n(u)|\,du}{((\Re \tilde{Z}_n^{(-1)}(x+ih_n)-u)^2+(\Im \tilde{Z}_n^{(-1)}(x+ih_n))^2)^2}\notag\\
&\le \int\limits_{[-2nN_n,2nN_n]}\frac{|V_n(u)|\,du}
{(\Re \tilde{Z}_n^{(-1)}(x+ih_n)-u)^2+(\Im \tilde{Z}_n^{(-1)}(x+ih_n))^2}.\notag
\end{align}
In view of (\ref{4.21}) and (\ref{4.34}), (\ref{4.35}), we deduce from this estimate 
the~following inequality
\begin{equation}\notag
\Big|\Im \int\limits_{[-2nN_n,2nN_n]}\frac{(\tilde{\mu}_{n}-\tilde{\nu}_{n})(du)}{\tilde{Z}_n^{(-1)}(x+ih_n)-u}\Big|
\le 
\begin{cases}
c(\mu)nN_n\varepsilon_n^*\quad\qquad\quad\quad\,\,\text{for}\quad |x|\le 5nN_n\\
c(\mu)n^2N_n\varepsilon_n^*/(1+x^2)\quad\text{for}\quad |x|>5nN_n
\end{cases}
\end{equation}
which implies the~upper bound
\begin{equation}\label{4.39}
\Big|\int\limits_{\mathbb R}e^{itx}\Im\int\limits_{[-2nN_n,2nN_n]}\frac{(\tilde{\mu}_{n}-\tilde{\nu}_{n})(du)}
{\tilde{Z}_n^{(-1)}(x+ih_n)-u}\,dx\Big|\le c(\mu)n^2N_n^2\varepsilon_n^*,\quad t\in\mathbb R.
\end{equation}
It remains to note that
\begin{equation}\label{4.40}
\int\limits_{\mathbb R}e^{itx}\Im\int\limits_{[-2N_n,2N_n]}\frac{(\tilde{\mu}-\tilde{\rho}_{n})(du)}{x-u+ih_n} \,dx= 
e^{-h_n|t|}(\varphi(t;\tilde{\mu})-\varphi(t;\tilde{\rho}_{n})),\quad t\in\mathbb R,
\end{equation}
where $\varphi(t;\tilde{\mu})$ and $\varphi(t;\tilde{\rho}_{n}))$ are characteristic functions of the~probability
measures $\tilde{\mu}$ and $\tilde{\rho}_{n}$,
respectively.
Recalling (\ref{4.21}), we finally conclude from (\ref{4.36}), (\ref{4.38})--(\ref{4.40}) that, 
for every fixed $v>0$,
\begin{equation}\label{4.41}
\int\limits_{[0,v]}|\varphi(t;\tilde{\mu})-\varphi(t;\tilde{\rho}_{n}))|\,dt\le c(\mu)n^{5/2}N_n^2e^{vh_n}\varepsilon_n^*.
\end{equation}

Since $\varepsilon_n^*\le 5n^3\varepsilon_n^{3/13}$, we see that, for $0<v<\frac{1}{260}\frac{\beta'}{\sqrt{m_2(\mu)}}$,
$$
n^{5/2}N_n^2 e^{vh_n}\varepsilon_n^*\le 5n^{11/2}\varepsilon_n^{1/13} e^{10 v\sqrt{m_2(\mu)}\sqrt{n}}\le 5n^{11/2}\varepsilon_n^{1/26}.
$$

Note that the~right-hand side of this inequality tends to zero as $n\to\infty$.
Since $\mu(\mathbb R\setminus[-2N_n,2N_n])\le\beta_3(\mu)/(2N_n)^3$, we obtain from (\ref{4.41}) the~upper bound
\begin{align}\label{4.42}
\int\limits_{[0,v]}(1-\Re\varphi(t;\tilde{\rho}_{n}))\,dt&\le \int\limits_{[0,v]}(1-\Re\varphi(t;\tilde{\mu}))\,dt
+c(\mu)n^{11/2}\varepsilon_n^{1/26}\notag\\
&\le \int\limits_{[0,v]}(1-\Re\varphi(t;\mu))\,dt+c(\mu)n^{11/2}\varepsilon_n^{1/26}
\end{align}
for every fixed $0<v<\frac{1}{260}\frac{\beta'}{\sqrt{m_2(\mu)}}$. 

Now we shall show that in the~sequence $\{\tilde{\rho}_{n}\}$ we can choose a~weakly convergent subsequence.
In order to prove this, we shall verify that for any $\varepsilon>0$ there exists $N(\varepsilon)>0$
such that
\begin{equation}\label{4.43}
\sup_{n}\tilde{\rho}_{n}(\mathbb R\setminus[-N(\varepsilon)),N(\varepsilon)])\le\varepsilon.
\end{equation}
It is obvious that (\ref{4.43}) follows from Proposition~\ref{3.3c} and (\ref{4.42}). We will denote a~weakly convergent
subsequence of measures $\{\tilde{\rho}_{n}\}$ by $\{\tilde{\rho}_{n}\}$ again. 
We see that $\{\tilde{\rho}_{n}\}$
converges weakly to 
an~probability measure $\rho$ and, by (\ref{4.41}),
\begin{equation}\notag
\varphi(t;\mu)=\varphi(t;\rho)\quad\text{for}\quad t\in[-\frac{\beta}{300\sqrt{m_2(\mu)}},\frac{\beta}{300\sqrt{m_2(\mu)}}].
\end{equation}. 
Thus, the~theorem is proved.

\section{ Probability measures $\mu$ with exponential decrease in $\Delta(\mu^{n\boxplus},\mathbf D^{\boxplus})$}

The aim of this section is to prove Theorem~2.9.
In the~sequel we need some results of the~classical moment problem (see\cite{Akh:1965}). R. Nevanlinna  established a~complete
description of all solutions for the~indeterminate Hamburger moment problem which we will briefly recall here. We use $P_k(z)$ and $Q_{k+1}(z)\,(Q_0(z)\equiv 0)$,
to denote  orthogonal polynomials of degree $k$ of the~first and second kinds, respectively, that corresponds to 
the~indeterminate moment problem $\{m_k(\mu)\}_{k=0}^{\infty}$. The~polynomials $P_k(z)$ and $Q_k(z)$ are real-valued on 
the~real line, and for them, the~series $\sum_{k=0}^{\infty}|P_k(z)|^2$ and $\sum_{k=0}^{\infty}|Q_k(z)|^2$ 
converge almost everywhere in $\mathbb C$. Define the~four functions
\begin{align}
A(z)&=z\sum_{k=0}^{\infty}Q_k(0)Q_k(z),\quad B(z)=-1+z\sum_{k=0}^{\infty}Q_k(0)P_k(z),\notag\\
C(z)&=1+z\sum_{k=0}^{\infty}P_k(0)Q_k(z),\quad E(z)=z\sum_{k=0}^{\infty}P_k(0)P_k(z).\notag
\end{align}
As Riesz M. has proved (see\cite{Akh:1965}, Chapter 3), the~functions $A(z), B(z), C(z)$, and $E(z)$ are entire functions of 
no more than minimal type of order 1. It is known that the~zeros of the~entire functions $A(z)t-C(z)$ and $B(z)t-E(z)$ are 
real alternate and of multiplicity one for any real $t$.

According to Nevanlinna's theorem (see  \cite{Akh:1965}), the~formula
\begin{equation}\label{7.1}
\int\limits_{\mathbb R}\frac{\mu(du)}{u-z}=-\frac{A(z)f(z)-C(z)}{B(z)f(z)-E(z)} 
\end{equation}
establishes a~one-to-one relationship between the~set of all solutions $\mu$ of the~moments problem
under consideration and the~set of all functions $f(z)$ of the~Nevanlinna class, completed by the~constant $\infty$.
Finite real constants $t$ belong to this class by definition. In this paper we are interested only in the~solutions of 
the~moment problem that corresponds to the~functions $f(z)\equiv t$. We denote these solutions by $\mu_t,\,t\in\mathbb R$.
These probability measures $\mu_t$ are discrete and have growth points at the~zeros of the~functions $B(z)t-E(z)$.

Consider the sequence $\alpha_s=\int\limits_{\mathbb R}x^{s}e^{-2|x|/L(|x|)}\,dx,\,s=0,1,\dots$, 
where $L(x)\ge 1, x\ge 0$, is a slowly varying function such that $\int\limits_{\mathbb R_+}\frac {dx}{(1+x)L(x)}<\infty$.
Without loss of generality we assume that $L(x)$ has continuous derivative. For simplicity we assume $\alpha_0=1$ as well.
By a result of Krein (see~\cite{Akh:1965}, Ch. 2), the~sequence $\{\alpha_s\}_{s=0}^{\infty}$ generates an~indeterminate
moment problem. 

Now introduce a symmetric
$\boxplus$-infinitely divisible probability measure $\mu$ with free cumulants $\alpha_s(\mu)=\alpha_s,\,s=1,2,\dots$. We see that 
that the~sequence $\{m_s(\mu)\}_{s=0}^{\infty}$ generates an~indeterminate
moment problem as well. Consider symmetric probability measures $\frac 12(\mu_t+\bar{\mu}_t),\,t\in\mathbb R$. These measures are discrete and therefore are not
$\boxplus$-infinitely divisible.

First evaluate the cumulants $\alpha_{2s}(\mu),\,s=1,\dots$. 

Since  $L(x)$ from Theorem~2.9  satisfies the functional equation $L(x)\sim L(x/L(x))$
as $x\to \infty$, we see from 
\cite{Se:1976}, p. 48--49, that
\begin{equation}\label{7.1a}
L(x)\sim L(x/L(x))\sim L(xL(x))\quad\text{as}\quad x\to \infty.
\end{equation}

\begin{lemma}\label{l7.1*}
The~following inequalities hold
\begin{equation}\notag
\alpha_{2s}(\mu)\le c(3sL(s))^{2s} e^{-3s/4},\quad s\in\mathbb N,
\end{equation}
with some constant $c$ depending on $L$ only.
\end{lemma}
\begin{proof}
It is well-known that $xL'(x)/L(x)=o(1)$ as $x\to \infty$, see~\cite{Se:1976}. Using this relation and 
(\ref{7.1a}) we conclude that the function $x/L(x)$ increases monotonously as $x\to\infty$ and 
the~equation $x/L(x)=2s$ has a~unique solution $x_s$ such that
$x_s\sim 2sL(2s)\sim 2sL(s)$ as $s\to\infty$. Hence solutions $y_s$ of the~equations 
\begin{equation}\notag
\frac x{L(x)}\Big(1-\frac {xL'(x)}{L(x)}\Big)=2s 
\end{equation}
satisfy the inequality $sL(s)\le y_s\le 3sL(s)$ for sufficiently large $s\ge s_0\ge 1$.
It remains to note that the~inequality
\begin{align}
\int\limits_{\mathbb R}x^{2s}e^{-2|x|/L(|x|)}\,dx&\le\max_{y_s}\Big(y_s^{2s}e^{-y_s/L(y_s)}\Big)\int\limits_{\mathbb R}
e^{-|x|/L(|x|)}\,dx\notag\\
&\le c(3sL(s))^{2s}e^{-sL(s)/L(s/L(s))}
\le c(3sL(s))^{2s}e^{-3s/4}
\end{align}
holds for $s\ge s_0$ with some constant $c$ depending on $L$ only. The~lemma is proved.
\end{proof}

Evaluate moments $m_s(\mu),\,s=2,3,\dots$.
\begin{lemma}\label{l7.1}
The~inequalities hold
\begin{equation}\label{7.1*}
\alpha_{2s}\le m_{2s}(\mu)\le 4^{2s}
\alpha_{2s},\quad s\in\mathbb N.
\end{equation}
\end{lemma}
\begin{proof}
Since $\alpha_{2j-1}(\mu)=\alpha_{2j-1}=0$ and $\alpha_{2j}(\mu)=\alpha_{2j}>0$ for $j=1,\dots$, the~left inequality in (\ref{7.1*}) obviously follows
from (\ref{3.8b}). Now we obtain from (\ref{3.8a}) and (\ref{3.8b})
\begin{equation}\label{7.1*a}
m_{2s}(\mu)\le \frac 1{2s}\sum_{l=1}^{2s}{{2s}\choose {l}}{{2s}\choose {l-1}}\prod_{j=1}^l\alpha_{|V_j|}(\mu)
=\frac 1{2s}\sum_{l=1}^{2s}{{2s}\choose {l}}{{2s}\choose {l-1}}\prod_{j=1}^l\alpha_{|V_j|},
\quad s\in\mathbb N.
\end{equation}
By Lyapunov's inequality, $\alpha_{|V_j|}\le  (\alpha_{2s})^{|V_j|/(2s)}$.
Therefore we have from (\ref{7.1*a})
\begin{equation}\notag
 m_{2s}(\mu)\le \alpha_{2s}\frac 1{2s}\sum_{l=1}^{2s}{{2s}\choose {l}}{{2s}\choose {l-1}}
\le \frac{\alpha_{2s}}{2s}4^{2s}. 
\end{equation}
The~right inequality in (\ref{7.1*}) follows immediately from this estimate.
\end{proof}

Using Lemma~\ref{l7.1} we estimate tails of the~measures $\mu$ and $\mu_t$.
\begin{lemma}\label{l7.2}
The~following upper bounds hold 
\begin{equation}\notag
\mu(\{|x|\ge N\})\le c(\mu)e^{-N/(48L(N))},\quad \mu_t(\{|x|\ge N\})\le c(\mu)e^{-N/(48L(N))},\quad N\ge 1. 
\end{equation}
\end{lemma}
\begin{proof}
By Lemmas~\ref{l7.1*} and~\ref{l7.1} we have
\begin{equation}\label{7.1*b}
m_{2s}(\mu)\le c(12sL(s))^{2s}e^{-3s/4},\quad s\in\mathbb N, 
\end{equation}
with some constant $c$ depending on $L$. It is easy follows from this bound that
$$
\mu(\{|x|>12sL(s)\})\le ce^{-3s/4}, \quad s\in\mathbb N.
$$ 
This inequality holds obviously
for all positive $s\ge 1$.
Using (\ref{7.1a}) we deduce
from the last inequality the~following estimate
\begin{equation}\notag
\mu(\{|x|>12s^*\})\le ce^{-s^*/2L(s^*)}  
\end{equation}
for all positive sufficiently large $s^*=sL(s), \,s\ge s_0$. The~assertion of the~lemma 
for the~measure $\mu$ follows immediately from this upper bound. One can prove the~lemma 
for the~measures $\mu_t$ in the~same way.
\end{proof}

Introduce the~segment $[-N_n, N_n]$ with $N_n:=\sqrt n/\tilde{c}(\mu)$ with sufficiently large constant $\tilde{c}(\mu)>0$ and 
consider probability measures $\tilde{\mu}$ and $\tilde{\mu_t}$ which are truncated on this segment.
By Lemma~\ref{l7.2} and Proposition~\ref{3.3b}, we have
\begin{equation}\label{7.2**}
\Delta(\mu_t^{n\boxplus},\mu^{n\boxplus})\le \Delta(\tilde{\mu_t}^{n\boxplus},\tilde{\mu}^{n\boxplus})
+ c(\mu)ne^{-N_n/(48L(N_n))},\quad n\in\mathbb N,
\end{equation}
where $c(\mu)>0$ does not depend on $\tilde{c}(\mu)$. In the sequel constants $c(\mu)$ do not depend on $\tilde{c}(\mu)$ as well.

In addition moments the~probability measures $\tilde{\mu}$ and $\tilde{\mu_t}$ are close in the~following sense.
\begin{lemma}\label{l7.3}
The following upper bounds hold
\begin{equation}\notag
|m_{2s}(\tilde{\mu})-m_{2s}(\tilde{\mu_t})|\le c(\mu)N_n^{2s}e^{-N_n/(96L(N_n))},\quad s=1,\dots,[N_n/(200L(N_n))].  
\end{equation} 
\end{lemma}
\begin{proof}
 By the definition of $\tilde{\mu}$ and $\tilde{\mu_t}$, we have
 $$
 |m_{2s}(\tilde{\mu})-m_{2s}(\tilde{\mu_t})|\le \int_{|x|\ge N_n}
 x^{2s}\,\mu(dx)+\int_{|x|\ge N_n} x^{2s}
 \,\mu_t(dx).
 $$
 The function $x^{2s-1}\exp\{-x/(96L(x))\},\,x\ge N_n$, is decreasing, therefore integrating by parts and using Lemma~\ref{l7.2}, we see that
 \begin{align}
\int_{|x|\ge N_n}x^{2s}\,\mu(dx)&\le
N_n^{2s}\mu(\{|x|\ge N_n\})+2s\int_{N_n}^{\infty}x^{2s-1}\,\mu(\{|x|\ge x\})\,dx
\notag\\
&\le c(\mu)N_n^{2s}e^{-N_n/(48L(N_n))}+
2c(\mu)s\int_{N_n}^{\infty}x^{2s-1}\,e^{-x/(48L(x))}\,dx\notag\\
&\le c(\mu)N_n^{2s}e^{-N_n/(96L(N_n))}.
\notag
 \end{align}
In the same way we obtain the bound
$$
\int_{|x|\ge N_n}x^{2s}\,\mu_t(dx)\le
c(\mu)N_n^{2s}e^{-N_n/(96L(N_n))}.
$$
The last two bounds imply the assertion of the lemma.
\end{proof}

As before in Section~4 $G_{\tilde{\mu}^{n\boxplus}}(z)=G_{\tilde{\mu}}(\tilde{Z_n}(z))$ and $G_{\tilde{\mu_t}^{n\boxplus}}(z)=G_{\tilde{\mu_t}}(\tilde{W_n}(z))$,
where $\tilde{Z_n}(z)$ and $\tilde{W_n}(z)$ belong to the~class $\mathcal F$ and satisfy the~functional equations
(\ref{4.4*}) and (\ref{4.22}) with $\tilde{\rho_n}=\tilde{\mu_t}$, respectively. By Lemma~3.9, we conclude that $\tilde{Z_n}(z)$ and $\tilde{W_n}(z)$ admit representations (\ref{4.5*}) and
(\ref{4.24}), respectively, where  the measures $\tilde{\sigma}_n$ and $\tilde{\tau}_n$ satisfy the conditions, for some $c(\mu)>0$,
\begin{equation}
\tilde{\sigma}_n(\mathbb R)=\tilde{\sigma}_n([-c(\mu)(n-1)N_n,c(\mu)(n-1)N_n]),\,
\tilde{\tau}_n(\mathbb R)=\tilde{\tau}_n([-c(\mu)(n-1)N_n,c(\mu)(n-1)N_n]). \notag
\end{equation}

Therefore, for $u\in[-N_n,N_n]$, 
\begin{equation}\notag
\frac 1{u-\tilde{Z_n}(z)}=\int\limits_{\mathbb R}\frac{\zeta(u;ds)}{s-z}\quad\text{and}\quad
\frac 1{u-\tilde{W_n}(z)}=\int\limits_{\mathbb R}\frac{\xi(u;ds)}{s-z},\quad z\in\mathbb C^+, 
\end{equation}
where $\zeta(u;ds)$ and $\xi(u;ds)$ are probability measures such that $\zeta(u;[-c(\mu)(n-1)N_n,c(\mu)(n-1)N_n])=1$ 
and $\xi(u;[-c(\mu)(n-1)N_n,c(\mu)(n-1)N_n])=1$ with some $c(\mu)>0$. Hence the~probability measures $\tilde{\mu}^{n\boxplus}$ and $\tilde{\mu_t}^{n\boxplus}$
are supported on the~segment $[-c(\mu)(n-1)N_n,c(\mu)(n-1)N_n]$.

Recall the~definition of the~domains $D_n(\tilde{\mu_t})=\{z=x+iy,x,y\in\mathbb R:y>\tilde{y}_t(x)\}$ and $D_n(\tilde{\mu})=\{z=x+iy, x,y\in\mathbb R:y>\tilde{y}(x)$. By Lemma~3.4 
we have the following lower bounds	
\begin{equation}\label{7.4}
|\tilde{Z_n}(z)|\ge c(\mu)\sqrt n,\quad  |\tilde{W_n}(z)|\ge c(\mu)\sqrt n,\quad z\in\mathbb C^+,\quad n\in\mathbb N.
\end{equation}

Using (\ref{4.4*}) and (\ref{4.22}) with $\tilde{\rho_n}=\tilde{\mu_t}$ we obtain the~relation, for $z\in\mathbb C^+$,
\begin{equation}\label{7.6}
-(n-1)(G_{\tilde{\mu}}(\tilde{Z_n}(z))-G_{\tilde{\mu_t}}(\tilde{W_n}(z)))=n(\tilde{Z_n}(z)-\tilde{W_n}(z))G_{\tilde{\mu}}(\tilde{Z_n}(z))G_{\tilde{\mu_t}}(\tilde{W_n}(z)).
\end{equation}
Since $G_{\tilde{\mu}}(\tilde{Z}_n(z))=\int_{[-N_n,N_n]}\frac{\tilde{\mu}(du)}{\tilde{Z}_n(z)-u}$ for $z\in\mathbb C^+$, by (\ref{7.4}), we see that
$|G_{\tilde{\mu}}(\tilde{Z}_n(z)|\le
c(\mu)/N_n$ for $\Im z\ge 0$. The same bound holds for $|G_{\tilde{\mu_t}}(\tilde{W}_n(z)|$. Therefore we conclude from (\ref{7.6}) that
\begin{equation}
 |G_{\tilde{\mu}}(\tilde{Z_n}(z))-G_{\tilde{\mu_t}}(\tilde{W_n}(z))|\le
 \frac {c(\mu)}{N_n^2}\,|\tilde{Z}_n(z)-
\tilde{W}_n(z)|, \quad \Im z\ge 0. 
\end{equation}

Now we need to estimate the~closeness of the~functions $\tilde{Z_n}(z)$ and $\tilde{W_n}(z)$.
\begin{lemma}\label{l7.5}
The~following upper bounds hold
\begin{equation}\label{7.5}
|\tilde{Z_n}(z)-\tilde{W_n}(z)|\le
c(\mu)e^{-c(\mu)\sqrt n/L(\sqrt{n})},
\quad \Im z\ge 0.
\end{equation}
\end{lemma}

In order to prove this lemma we need
some auxiliary results.
\begin{lemma}\label{l7.6}
The upper bound holds, for $\Im z\ge 0,\,|z|\ge 100N_n$,
\begin{equation}
|\tilde{Z_n}^{(-1)}(z)-\tilde{W_n}^{(-1)}(z)|\le c(\mu)(n-1)N_n e^{-N_n/(96L(N_n))}. \label{lt.6.1}
\end{equation}
\end{lemma}
\begin{proof} Note that, for $z\in\mathbb C^+$,
\begin{equation}
 \tilde{Z_n}^{(-1)}(z)-\tilde{W_n}^{(-1)}(z)=(n-1)(F_{\tilde{\mu}_t}(z)-F_{\tilde{\mu}}(z))=\frac{n-1} {G_{\tilde{\mu}_t}(z)G_{\tilde{\mu}}(z)}(G_{\tilde{\mu}}(z)-G_{\tilde{\mu}_t}(z)).\label{7.13}
\end{equation}
On the other hand we see that,  for  $l_1=[N_n/(200L(N_n))]$,
\begin{align}
G_{\tilde{\mu}}(z)-G_{\tilde{\mu}_t}(z))&=\frac{m_2(\tilde{\mu})-m_2(\tilde{\mu}_t)}{z^3}+\dots+\frac{m_{2(l_1-1)}(\tilde{\mu})-m_{2(l_1-1)}(\tilde{\mu}_t)}{z^{2l_1-1}}\notag\\
&+\frac{1}{z^{2l_1}}\int_{\mathbb R}\frac{ u^{2l_1}(\tilde{\mu}-\tilde{\mu}_t)\,(du)}{z-u}.\label{5.14}
\end{align}
By Lemma~5.4 and (\ref{7.1*b}), we conclude that
\begin{align}
 &|G_{\tilde{\mu}}(z)-G_{\tilde{\mu}_t}(z)|\le \frac{c(\mu)}{|z|^2}e^{-N_n/(96L(N_n))}\sum_{s=1}^{l_1}\frac{N_n^{2s}}{(100N_n)^{2s-1}}
 +\frac{c(\mu)}{|z|^2}\frac{m_{2l_1}(\mu)}{(100N_n)^{2l_1-1}}\notag\\
 &\le \frac{c(\mu)}{|z|^2}\Big(N_ne^{-N_n/(96L(N_n))}+N_n\Big(\frac{12l_1L(l_1)}{100N_n}\Big)^{2l_1}\Big)\le\frac{c(\mu)}{|z|^2}N_ne^{-N_n/(96L(N_n))}.\label{7.15}
\end{align}
Applying this bound and the lower bounds $|G_{\tilde \mu_t}(z)|\ge c(\mu)/|z|,\,|G_{\tilde \mu}(z)|\ge
c(\mu)/|z|$, for $|z|\ge 2N_n$, to (\ref{7.13})
we arrive at the assertion of the lemma. 
\end{proof}

\begin{lemma}\label{l7.7}
The following upper bound holds, for $z\in\mathbb C^+,\,|z|\ge 100N_n$,
\begin{equation}
\frac 1{\Im z}\Big|\Im\Big(\frac 1{G_{\tilde{\mu}}(z)}-\frac
1{G_{\tilde{\mu}_t}(z)}\Big)\Big|\le c(\mu)N_n e^{-N_n/(96L(N_n))}. \notag
\end{equation}
\end{lemma}
\begin{proof} 
We will use the following simple estimate
\begin{align}
&\frac 1{\Im z}\Big|\Im\Big(\frac 1{G_{\tilde{\mu}}(z)}-\frac
1{G_{\tilde{\mu}_t}(z)}\Big)\Big|\le
\frac 1{\Im z}|\Im (G_{\tilde{\mu}}(z)-G_{\tilde{\mu}_t}(z))|\frac 1{|
G_{\tilde{\mu}}(z)G_{\tilde{\mu}_t}(z))|}\notag\\
&+|G_{\tilde{\mu}}(z)-G_{\tilde{\mu}_t}(z)|\frac 1{\Im z}
\Big(\frac{\Im G_{\tilde{\mu}}(z)}{|G_{\tilde{\mu}}(z)|^2}
\frac 1{|G_{\tilde{\mu}_t}(z)|}+
\frac{\Im G_{\tilde{\mu}_t}(z)}{|G_{\tilde{\mu}_t}(z)|^2}
\frac 1{|G_{\tilde{\mu}}(z)|}\Big).\label{7.16}
\end{align}
Using the formula (\ref{7.14}) and repeating the argument of the proof of Lemma~5.6, we see that the first term
of the sum on the right hand-side of (\ref{7.16}) does not exceed $c(\mu)N_ne^{-N_n/(96L(N_n))}$.
Noting that $\frac 1{\Im z}|\Im G_{\tilde{\mu}}(z)|\le\frac 4{|z|^2}$ and
$\frac 1{\Im z}|\Im G_{\tilde{\mu}_t}(z)|\le\frac 4{|z|^2}$, for $|z|\ge 2N_n$, and using (\ref{7.15}),we obtain that the second
term in the sum on the right hand-side of (\ref{7.16}) does not exceed 
the quantity  $\frac{c(\mu)}{|z|}N_ne^{-N_n/(96L(N_n))}$. The assertion of the lemma follows from the last two bounds.
\end{proof}

Recalling Proposition~3.1 we see that\begin{equation}
\tilde{Z}_n^{(-1)}(z)=z-(n-1)\int_{[-N_n,N_n]}\frac{\tilde{\sigma}(du)}{u-z}\quad \text{and}\quad                                     
\tilde{W}_n^{(-1)}(z)=z-(n-1)\int_{[-N_n,N_n]}\frac{\tilde{\sigma}_t(du)}{u-z}, \notag                                    \end{equation}
where $\tilde{\sigma}$ and $\tilde{\sigma}_t$ are non negative finite measures. Denote by $M_n(\tilde{\mu})\ge N_n$ and $M_n(\tilde{\mu}_t)\ge N_n$ positive numbers satisfying the equations
\begin{equation}
 \int_{-[N_n,N_n]}\frac{\tilde{\sigma}(du)}{(u+M_n(\tilde{\mu}))^2}=\frac 1{n-1}\quad\text{and}\quad
 \int_{[-N_n,N_n]}\frac{\tilde{\sigma}_t(du)}{(u+M_n(\tilde{\mu}_t))^2}=\frac 1{n-1},\notag
\end{equation}
respectively. Without loss of generality we assume that $M_n(\tilde{\mu})\ge M_n(\tilde{\mu}_t)$. 
Now we define a Jordan curve ${\gamma}_n(\tilde{\mu})= \{y=y_n(x;\tilde{\mu}),\, x\in\mathbb R\}$, 
where the function $y_n(x;\tilde{\mu})$ is the solution of the functional equation
\begin{equation}
\int_{[-N_n,N_n]}\frac{\tilde{\sigma}(du)}{(u-x)^2+y_n(x;\tilde{\mu})^2}=\frac 1{n-1} \notag
\end{equation}
for $x\in[-M_n(\tilde{\mu}),M_n(\tilde{\mu})]$.
For $|x|>M_n(\tilde{\mu})$ $y_n(x;\tilde{\mu})=0$.
In the same way we define the curve
$\gamma_n(\tilde{\mu}_t)=\{y=y_n(x;\tilde{\mu}_t),\,x\in\mathbb R\}$.

\begin{lemma}
The following bound holds
 \begin{equation}
  |y_n(x;\tilde{\mu})-y_n(x;\tilde{\mu}_t)|\le c(\mu)e^{-N_n/(200L(N_n))},\quad x\in\mathbb R.\notag
 \end{equation}
 \end{lemma}
\begin{proof}
 Assume that $-M_n(\tilde{\mu}_t)< x\le 0$. In this case $y_n(x;\tilde{\mu})>0$ and  $y_n(x;\tilde{\mu}_t)>0$, and we have the relation
 \begin{align}
  (y_n(x;\tilde{\mu}_t)^2&-y_n(x;\tilde{\mu})^2)\int_{[-N_n,N_n]}\frac{\tilde{\sigma}(du)}{((u-x)^2+
  y_n(x;\tilde{\mu})^2)((u-x)^2+y_n(x;\tilde{\mu}_t)^2)}\notag\\
  &=\frac{1}{y_n(x;\tilde{\mu}_t)}\Im\Big(\frac{1}{G_{\tilde{\mu}_t}(x+y_n(x;\tilde{\mu}_t))}-\frac{1}{G_{\tilde{\mu}}(x+y_n(x;\tilde{\mu}_t))}\Big).\notag
 \end{align}
Since $x^2+y_n(x;\tilde{\mu}_t))^2\ge
c(\mu)n$ we obtain, using Lemma~5.7,
\begin{align}
|y_n(x;\tilde{\mu}_t)&-y_n(x;\tilde{\mu})|^2\int_{[-N_n,N_n]}
\frac{\tilde{\sigma}(du)}{((u-x)^2+
  y_n(x;\tilde{\mu})^2)((u-x)^2+y_n(x;\tilde{\mu}_t)^2)}\notag\\
  &\le c(\mu)\sqrt ne^{-N_n/(96L(N_n))}.\notag
\end{align}
Since the integral in this inequality
is greater or equal to $c(\mu)/n^2$, we finally get the upper bound
$|y_n(x;\tilde{\mu}_t) -y_n(x;\tilde{\mu})|\le c(\mu)e^{-N_n/(200L(N_n))}$ for $-M_n(\tilde{\mu}_t)< x\le 0$. Since $y_n(-M_n(\tilde{\mu}_t);\tilde{\mu}_t)=0$, we have $y_n(-M_n(\tilde{\mu}_t);\tilde{\mu})\le c(\mu)e^{-N_n/(200L(N_n))}$. The function $y_n(x;\tilde{\mu})$ is increasing for $-M_n(\tilde{\mu})\le x\le -M_n(\tilde{\mu}_t)$, hence
$y_n(x;\tilde{\mu})\le c(\mu)e^{-N_n/(200L(N_n))}$ for such $x$.
The same arguments hold for $x>0$. 
The lemma is proved. 
\end{proof}

\begin{lemma}
The function $y_n(x;\tilde{\mu}_t)$
is differentiable for $|x|<M_n(\tilde{\mu}_t)$ with the derivative satisfying the relation
\begin{equation}
 y_n(x;\tilde{\mu}_t)y_n'(x;\tilde{\mu}_t)=\int_{-[N_n,N_n]}\frac{(u-x)\,\tilde{\sigma}_t(du)}{((u-x)^2+y_n(x;\tilde{\mu}_t)^2)^2}\Big/\int_{[-N_n,N_n]}\frac{\tilde{\sigma}_t(du)}{((u-x)^2+y_n(x;\tilde{\mu}_t)^2)^2}.\notag
\end{equation}
\end{lemma}

\begin{proof}
The proof is simple and we omit it. 
\end{proof}

Now we prove Lemma~5.5.

\begin{proof}
Consider the functions 
\begin{align}
f(x)&=Z_{n}^{(-1)}(x+iy_n(x;\tilde{\mu})),\,\,\,\,|x|<M_n(\tilde{\mu}),\notag\\
f_t(x)&=W_{n}^{(-1)}(x+iy_n(x;\tilde{\mu}_t)),\,|x|<M_n(\tilde{\mu}_t).\notag
\end{align}
We see that, for $|x|<M_n(\tilde{\mu}_t)$,
\begin{align}
 f_t'(x)&=1+(n-1)\int_{[-N_n,N_n]}\frac{\tilde{\sigma}_t(du)}{(u-x)^2+y_n(x;\tilde{\mu}_t)^2}\notag\\
 &+(n-1)\int_{[-N_n,N_n]}(u-x)\frac{-2(u-x)+2y_n(x;\tilde{\mu}_t)y_n'(x;\tilde{\mu}_t)}{((u-x)^2+y_n(x;\tilde{\mu}_t)^2))^2}\,\tilde{\sigma}_t(du)\notag\\
 &=2(n-1)\Big(
 \int_{[-N_n,N_n]}\frac{y_n(x;\tilde{\mu}_t)^2\,\tilde{\sigma}_t(du)}{((u-x)^2+
 y_n(x;\tilde{\mu}_t)^2)^2}\notag\\
 &+\Big(\int_{[-N_n,N_n]}
 \frac{(u-x)\,\tilde{\sigma}_t(du)}{((u-x)^2+y_n(x;\tilde{\mu}_t)^2)^2}\Big)^2
 \Big/\int_{[-N_n,N_n]}\frac{\tilde{\sigma}_t(du)}{((u-x)^2+y_n(x;\tilde{\mu}_t)^2)^2}\Big).
\end{align}

We easily conclude from (5.17) that
\begin{equation}
 c_7(\mu)
 \le f_t'(x)\le c_8(\mu),\quad |x|<M_n(\tilde{\mu}_t).
\end{equation}

Fix $s\in\mathbb R$ and real $x$ and $\tilde x$ such that $s=f(x)=f_t(\tilde x)$. Denote as well
$\tilde s=f_t(x)$. 

We see that
\begin{align}
 s-\tilde s&=f(x)-f_t(x)=Z_n^{(-1)}(x+iy_n(x;\tilde{\mu}))-W_n^{(-1)}(x+iy_n(x;\tilde{\mu}_t))\notag\\
 &=Z_n^{(-1)}(x+iy_n(x;\tilde{\mu}))-W_n^{(-1)}(x+iy_n(x;\tilde{\mu}))\notag\\&+W_n^{(-1)}(x+iy_n(x;\tilde{\mu}))-W_n^{(-1)}(x+iy_n(x;\tilde{\mu}_t)).\label{7.17*}
\end{align}
Since, for $z_1,\,z_2\in\mathbb C^+$,
\begin{equation}
W_n^{(-1)}(z_1)-W_n^{(-1)}(z_2)=
z_1-z_2+(n-1)\int_{[-N_n,N_n]}\frac{(z_2-z_1)\,\tilde{\sigma}_t(du)}{(u-z_1)(u-z_2)},\notag
\end{equation}
we obtain the upper bound, for $|z_1|\ge 2N_n$ and $|z_2|\ge 2N_n$,
\begin{equation}
|W_n^{(-1)}(z_1)-W_n^{(-1)}(z_2)|
\le c(\mu)|z_1-z_2|.\notag
\end{equation}
Therefore we have
\begin{equation}
|W_n^{(-1)}(x+iy_n(x;\tilde{\mu}))-W_n^{(-1)}(x+iy_n(x;\tilde{\mu}_t))|\le c(\mu)|y_n(x;\tilde{\mu})-
y_n(x;\tilde{\mu}_t)| \notag
\end{equation}
and, by Lemma~5.8, we get finally
\begin{equation}\label {7.18}
|W_n^{(-1)}(x+iy_n(x;\tilde{\mu}))-W_n^{(-1)}(x+iy_n(x;\tilde{\mu}_t))|\le
c(\mu)e^{-N_n/(200L(N_n))}.
\end{equation}
By Lemma~5.6, we obtain as well
\begin{equation}\label {7.19}
 |Z_n^{(-1)}(x+iy_n(x;\tilde{\mu}))-W_n^{(-1)}(x+iy_n(x;\tilde{\mu}))|\le 
 c(\mu)e^{-N_n/(200L(N_n))}.
\end{equation}
Therefore, applying (\ref{7.18}) and (\ref{7.19}) to (\ref{7.17*})
, we arrive at the desired estimate
\begin{equation}
|f_t(x)-f_t(\tilde x)|= |s-\tilde s|\le c(\mu)e^{-N_n/(200L(N_n))}.  
\end{equation}

Since $f_t(x)-f_t(\tilde x)=f_t'(x+\theta (\tilde x-x))(\tilde x-x)$, we 
have in view of (5.18) $|\tilde x-x|\le c(\mu)e^{-N_n/(100L(N_n))}|$.

Now we estimate $|y_n(x;\tilde{\mu}_t)^2-y_n(\tilde x;\tilde{\mu}_t)^2|$, using the formula
\begin{equation}
 y_n(x;\tilde{\mu}_t)^2-y_n(\tilde x;\tilde{\mu}_t)^2=(y_n(x;\tilde{\mu}_t)^2)'\big|_{x=x^*}(\tilde x-x),\label {7.22}
\end{equation}
where $x^*=x+\theta(\tilde x-x)$.
By Lemma~5.9, we see that $|(y_n(x;\tilde{\mu}_t)^2)'|\le c(\mu)n^{1/2}$ for $|x|<M_n(\tilde{\mu}_t)$.
If $y_n(x;\tilde{\mu}_t)+y_n(\tilde x;\tilde{\mu}_t)\ge e^{-N_n/(400L(N_n))}$, we obtain from (\ref{7.22}) that $|y_n(x;\tilde{\mu}_t)-y_n(\tilde x;\tilde{\mu}_t)|\le
c(\mu)e^{-N_n/(400L(N_n))}$. The same bound holds of cause in the case $y_n(x;\tilde{\mu}_t)+y_n(\tilde x;\tilde{\mu}_t)< e^{-N_n/(400L(N_n))}$.

Since 
\begin{equation}
 |x+iy_n(x;\tilde{\mu}_0)-\tilde x-iy_n(\tilde x;\tilde{\mu}_t)|\le |x-\tilde x|+|y_n(x;\tilde{\mu}_0)-y_n( x;\tilde{\mu}_t)|+|y_n(x;\tilde{\mu}_t)-y_n(\tilde x;\tilde{\mu}_t)|,
 \notag
\end{equation}
we finally conclude that
\begin{equation}
 |x+iy_n(x;\tilde{\mu}_0)-\tilde x-iy_n(\tilde x;\tilde{\mu}_t)|\le
 c(\mu)e^{-N_n/(400L(N_n))}
 \notag
\end{equation}
and we have the estimate $|Z_{n}(s)-W_{n}(s)|\le c(\mu)e^{-N_n/(400L(N_n))}$. Lemma~5.5 is proved.
\end{proof}

Now we can estimate $\Delta(\tilde{\mu}^{n\boxplus},\tilde{\mu_t}^{n\boxplus})$ using the~inversion formula.
Recalling that $\tilde{\mu}^{n\boxplus}$ and $\tilde{\mu_t}^{n\boxplus}$ are supported on the segment $[-c(\mu)(n-1)N_n,c(\mu)(n-1)N_n]$, we 
have, for $x$ from this segment, 
\begin{align}
\tilde{\mu}^{n\boxplus}((-c(\mu)(n-1)N_n,x))&-\tilde{\mu_t}^{n\boxplus}([-c(\mu)(n-1)N_n,x))\notag\\
&=-\lim_{\eta\to 0}\frac 1{\pi}
\int\limits_{[-c(\mu)(n-1)N_n,x)}
\Im ( G_{\tilde{\mu}}(\tilde{Z_n}(u+i\eta))-G_{\tilde{\mu_t}}(\tilde{W_n}(u+i\eta))\,du.\notag
\end{align}
Applying to this formula estimate (5.9)
and (5.10), we get the~upper bound
\begin{equation}\label{7.14} 
\Delta(\tilde{\mu}^{n\boxplus},\tilde{\mu_t}^{n\boxplus})\le c(\mu)nN_ne^{-c(\mu)N_n/L(N_n)}.
\end{equation}
The~statement of the~theorem follows immediately from (\ref{7.2**}) and (\ref{7.14}).

\section{ Probability measures $\mu$ with power decrease in $\Delta(\mu^{n\boxplus},\mathbf D^{\boxplus})$}

As the~first step we prove the~following result.
\begin{theorem}\label{th6.1}
Let $\mu$ and $\rho$ belong to class $\mathcal M_{2k}$, where $k\in\mathbb N_1$, and let
$m_j(\mu)=m_j(\rho),\,j=1,\dots,2k-1$.
Then
\begin{equation}\notag
\Delta(\mu^{n\boxplus},\rho^{n\boxplus})\le c(\mu,\rho)n^{-(k-4)/4},\quad n\in\mathbb N. 
\end{equation} 
\end{theorem}

We prove this theorem, using arguments of Section~5.
\begin{proof}
 Introduce the~segment $[-N_n, N_n]$ with $N_n:=\sqrt n/\tilde{c}(\mu,\rho)$ with sufficiently large constant $\tilde{c}(\mu,\rho)>0$ and 
consider truncated on this segment probability measures $\tilde{\mu}$ and $\tilde{\rho}$.
By the bounds
$$
\mu(\{|x|\ge N\})\le\frac{m_{2k}(\mu)}{N^{2k}},\quad
\mu(\{|x|\ge N\})\le\frac{m_{2k}(\rho)}{N^{2k}}
$$
for $N\ge 1$,
and Proposition~\ref{3.3b}, we have
\begin{equation}\label{7.2***}
\Delta(\rho^{n\boxplus},\mu^{n\boxplus})\le \Delta(\tilde{\rho}^{n\boxplus},\tilde{\mu}^{n\boxplus})
+ c(\mu,\rho)N_n^{-2k+1},\quad n\in\mathbb N,
\end{equation}
where $c(\mu,\rho)>0$ does not depend on $\tilde{c}(\mu,\rho)$. In the sequel constants $c(\mu,\rho)$ do not depend on $\tilde{c}(\mu,\rho)$ as well.

As before in Section~5 $G_{\tilde{\mu}^{n\boxplus}}(z)=G_{\tilde{\mu}}(\tilde{Z_n}(z))$ and $G_{\tilde{\rho}^{n\boxplus}}(z)=G_{\tilde{\rho}}(\tilde{W_n}(z))$,
where $\tilde{Z_n}(z)$ and $\tilde{W_n}(z)$ are belonging to the~class $\mathcal F$ and satisfy the~functional equations
(\ref{4.4*}) and (\ref{4.22}) with $\tilde{\rho_n}=\tilde{\rho}$, respectively. By Lemma~3.9, we conclude that $\tilde{Z_n}(z)$ and $\tilde{W_n}(z)$ admit representations (\ref{4.5*}) and
(\ref{4.24}), respectively, where  the measures $\tilde{\sigma}_n$ and $\tilde{\tau}_n$ satisfy the conditions, for some $c(\mu,\rho)>0$,
\begin{align}
&\tilde{\sigma}_n(\mathbb R)=\tilde{\sigma}_n([-c(\mu,\rho)(n-1)N_n,c(\mu,\rho)(n-1)N_n]),\notag\\
&\tilde{\tau}_n(\mathbb R)=\tilde{\tau}_n([-c(\mu,\rho)(n-1)N_n,c(\mu,\rho)(n-1)N_n]). \notag
\end{align}

As in Section~5 we can show that the~probability measures $\tilde{\mu}^{n\boxplus}$ and $\tilde{\rho}^{n\boxplus}$
are supported on the~segment $[-c(\mu,\rho)(n-1)N_n,c(\mu,\rho)(n-1)N_n]$.
\end{proof}

As in Section~5 we
easily obtain the analogue of (5.10)
\begin{equation}
 |G_{\tilde{\mu}}(\tilde{Z_n}(z))-G_{\tilde{\rho}}(\tilde{W_n}(z))|\le
 \frac {c(\mu,\rho)}{N_n^2}\,|\tilde{Z}_n(z)-
\tilde{W}_n(z)|, \quad z\in\mathbb R. 
\end{equation}

Now we need to obtain
an analogue of Lemma~5.5.
\begin{lemma}\label{l7.5*}
The~following upper bounds hold
\begin{equation}\label{7.5*}
|\tilde{Z_n}(z)-\tilde{W_n}(z)|\le
\frac{c(\mu,\rho)}{n^{(k-2)/4}},
\quad \Im z\ge 0.
\end{equation}
\end{lemma}

In order to prove this lemma we need
some auxiliary results
which are analogues of Lemmas~5.6, 5.7.
\begin{lemma}\label{l7.6**}
The upper bound holds, for $z\in\mathbb C^+,\,|z|\ge 2N_n$,
\begin{equation}
|\tilde{Z_n}^{(-1)}(z)-\tilde{W_n}^{(-1)}(z)|\le \frac{c(\mu,\rho)(n-1)}{N_n^{2k-1}}. 
\end{equation}
\end{lemma}
\begin{lemma}\label{l7.7**}
The upper bound holds, for $z\in\mathbb C^+,\,|z|\ge 2N_n$,
\begin{equation}
\frac 1{\Im z}\Big|\Im\Big(\frac 1{G_{\tilde{\mu}}(z)}-\frac
1{G_{\tilde{\rho}}(z)}\Big)\Big|\le \frac{c(\mu,\rho)}{N_n^{2k-1}}. \notag
\end{equation}
\end{lemma}
The proof of these two lemmas we obtain in the way as the proof of Lemmas~5.6,5.6.

Furthermore we repeat the argument of Section~5, replacing the measures $\tilde{\mu}$ and $\tilde{\rho}$ by the measures of $\tilde{\mu}$ and $\tilde{\mu_t}$ and the functions $y_n(x;\tilde{\mu}),\,y_n(x;\tilde{\rho})$ by $y_n(x;\tilde{\mu}),\,y_n(x;\tilde{\mu_t})$. We use in these arguments the constants $c(\mu,\rho)$ instead of $c(\mu)$.

In the same way as in the proof of Lemma~5.8
we prove 
\begin{lemma}
 The following  bound holds
 $$
 |y_n(x;\tilde{\mu})-y_n(x;\tilde{\rho})|\le\frac{c(\mu,\rho)}{N_n^{(2k-1)/2}}.
 $$
\end{lemma}

Using Lemmas 6.3, 6.5 we prove Lemma 6.2 in the way as in Section 5 we proved Lemma 5.5.

Now we estimate $\Delta(\tilde{\mu}^{n\boxplus},\tilde{\rho}^{n\boxplus})$ using the~inversion formula.
Recalling that $\tilde{\mu}^{n\boxplus}$ and $\tilde{\rho}^{n\boxplus}$ are supported on the segment $[-c(\mu,\rho)(n-1)N_n,c(\mu,\rho)(n-1)N_n]$, we 
have, for $x$ from this segment, 
\begin{align}
\tilde{\mu}^{n\boxplus}((-c(\mu,\rho)(n-1)N_n,x))&-\tilde{\rho}^{n\boxplus}([-c(\mu)(n-1)N_n,x))\notag\\
&=-\lim_{\eta\to 0}\frac 1{\pi}
\int\limits_{[-c(\mu,\rho)(n-1)N_n,x)}
\Im ( G_{\tilde{\mu}}(\tilde{Z_n}(u+i\eta))-G_{\tilde{\rho}}(\tilde{W_n}(u+i\eta))\,du.\notag
\end{align}
Applying to this formula the estimates (6.2)
and (6.3), we get the~upper bound
\begin{equation}\label{7.14*} 
\Delta(\tilde{\mu}^{n\boxplus},\tilde{\rho}^{n\boxplus})\le c(\mu,\rho)nN_n\frac 1{N_n^2}\frac 1{n^{(k-2)/4}}\le \frac{c(\mu.\rho)}{n^{(k-4)/4}}.
\end{equation}
The~statement of the~theorem follows immediately from (\ref{7.2***}) and (\ref{7.14*}).


As the~second step we prove the~following result, using arguments of Section~4.

\begin{theorem}\label{th6.6}
Let $\mu\in\mathcal M_{2k}$, where $k\in\mathbb N_1$, and let $\alpha_{2k}(\mu)< 0$.
Let there exist a $\boxplus$-infinitely divisible probability measure $\rho$ such that $m_{2k}(\rho)<\infty$ and
$m_j(\mu)=m_j(\rho),\,j=1,\dots,2k-1$. Then
\begin{equation}\label{6.1}
\Delta(\mu^{n\boxplus},\rho^{n\boxplus})\ge c(\mu,\rho)n^{-\frac{k^2+3k+1}{k-1}},\quad n\in\mathbb N. 
\end{equation} 
\end{theorem}
\begin{proof}
Assume that (\ref{6.1}) does not hold. Then there exists a subsequence of increasing positive integers ${n_l}\}_{l=1}^{\infty}$ such that
\begin{equation}\label{6.2}
 \Delta(\mu^{n_l\boxplus},\rho^{n_l\boxplus})\le \varepsilon_{n_l}:=o(1)\,n_l^{-\frac{k^2+3k+1}{k-1}},\quad n_l\in\mathbb N.
\end{equation}

Our nearest aim is to show that (\ref{6.2}) implies $m_{2k}(\mu)=m_{2k}(\rho)$.

In the sequel we denote $n_l$ by $n$. As in Section~4 we denote 
by $\tilde{\mu}$ and $\tilde{\rho}$ the truncated probability measures of $\mu$ and $\rho$, respectively, with the parameter $N_n=(n/\varepsilon_n)^{1/(2k)}$. We denote
in this section $\varepsilon_n^*=
\varepsilon_n+n/N_n^{2k}$. Instead
of Lemma~4.3 we will use the estimate $\rho(|u|>2N_n)\le m_{2k}(\rho)/(2N_n)^{2k}$. Then repeating the argument of Section~4 we arrive at the following upper bound, for every $v>0$,
\begin{equation}\notag
\int_{[0,v]}|\varphi(t;\tilde{\mu})-\varphi(t;\tilde{\rho})|\,dt\le\big(c(\mu)+c(\rho)\big)n^{5/2}N_n^2e^{vh_n}\varepsilon_n^*. 
\end{equation}
Since
\begin{align}
\varphi(t; \tilde{\mu})-\varphi(t;
\tilde{\rho})&=\varphi(t;\mu)-\varphi(t;\rho)+\mu(|u|>2N_n)
-\rho(|u|>2N_n)\notag\\
&-\int_{|u|>2N_n}e^{itu}\,\mu(du)+\int_{|u|>2N_n}
e^{itu}\,\rho(du),\quad t\in\mathbb R,\notag
\end{align}
we easily obtain
\begin{align}
 \big|\int_{[0,v]}(\varphi(t;\mu)-
 \varphi(t;\rho))\,dt\big|&\le\int_{[0,v]}|\varphi(t; \tilde{\mu})-\varphi(t;\tilde{\rho})|\,dt+\frac{m_{2k}(\mu)+m_{2k}(\rho)}{(2N_n)^{2k}}
 \notag\\
 &\le(c(\mu)+c(\rho))n^{5/2}N_n^2 e^{vh_n}\varepsilon_n^*.\notag
\end{align}
By Taylor's formula
\begin{equation}
 \varphi(t;\mu)-\varphi(t;\rho)=
 \frac{m_{2k}(\mu)-m_{2k}(\rho)}{2k!}t^{2k}+o(t^{2k})\notag
\end{equation}
for small $t>0$, and we obtain from the preceding inequality with $v=1/\sqrt n$ the relation
$m_{2k}(\mu)-m_{2k}(\rho)=o(1)$ as $n\to\infty$. Thus, we proved that
$m_{2k}(\mu)=m_{2k}(\rho)$. It remains to note that then
$\alpha_j(\mu)=\alpha_j(\rho)$ 
for $j=1,\dots , 2k$. Since $\alpha_{2k}(\rho)\ge 0$, we arrive at the contradiction.
Hence (\ref{6.1}) is proved.

\end{proof}

Theorem~2.10 follows from Theorems~6.1 and 6.6.

\end{document}